\documentclass[12pt]{article}
\usepackage[a4paper, total={6in,10in}]
{geometry}
\usepackage[utf8]{inputenc}
\usepackage{fullpage}
\usepackage{amsmath}
\usepackage{amssymb}
\usepackage{amsfonts}
\usepackage{amsthm}
\usepackage{cite}
\usepackage{algorithm2e}
\usepackage{enumitem}
\newtheoremstyle{dotless}{}{}{\itshape}{}{\bfseries}{}{ }{}
\theoremstyle{dotless}
\newtheorem{theorem}{Theorem}[section]
\newtheorem{corollary}{Corollary}[section]
\newtheorem{lemma}{Lemma}[section]

\theoremstyle{definition}
\newtheorem{example}{Example}

\usepackage{graphicx}
\usepackage{xcolor}

\DeclareMathOperator{\adj}{adj}

\title{$A_{\alpha}$-Spectra of $Q$- and $T$-Join Graphs with Applications to Cospectral Constructions}

\author{Mainak Basunia\thanks{Department of Mathematics, Indian Institute of Technology Kharagpur, Kharagpur 721302, India. Email: leo28mynnix@gmail.com, mainakmaths@iitkgp.ac.in}\and Pratima Panigrahi\thanks{Department of Mathematics, Indian Institute of Technology Kharagpur, Kharagpur 721302, India. Email: pratima@maths.iitkgp.ac.in}}

\date{}

\begin{document}
\maketitle
\baselineskip=0.25in

\begin{abstract}\label{sec1}
\noindent For $\alpha \in [0,1]$, the $A_{\alpha}$-matrix of a graph $G$ is defined by $A_{\alpha}(G) = \alpha D(G) + (1- \alpha) A(G)$, where $A(G)$ and $D(G)$ denote the adjacency matrix and the diagonal degree matrix of $G$, respectively. In this paper, we study the $A_{\alpha}$-characteristic polynomials and $A_{\alpha}$-spectra of graphs obtained via four recently introduced join operations, namely the $Q$-vertex join, $Q$-edge join, $T$-vertex join, and $T$-edge join, applied to two graphs $G_1$ and $G_2$. We derive explicit expressions for the $A_{\alpha}$-characteristic polynomials of these constructions when the first factor graph is regular. Furthermore, we determine the complete $A_{\alpha}$-spectra of these graphs in terms of the $A_{\alpha}$-spectra of the factor graphs, particularly when the second factor graph is regular or complete bipartite. The significance of these results lies in the fact that they enable efficient computation of the $A_{\alpha}$-spectra of large complex graphs arising from these joins, directly from the $A_{\alpha}$-spectra of the smaller constituent graphs, without explicitly constructing and handling the complex $A_{\alpha}$-matrices of those large graphs. Finally, as an application, we demonstrate how to construct infinitely many families of non-isomorphic graphs that are $A_{\alpha}$-cospectral.


\end{abstract}

\section{Introduction}\label{sec2}
In this article, we consider all graphs undirected and simple. Let $G$ be a graph of order $n$ and size $m$ with vertex set $V(G)=\{ v_1, v_2, \ldots, v_n \}$ and edge set $E(G)=\{ e_1, e_2, \ldots, e_m \}$. The adjacency matrix $A(G)$ of the graph $G$ is a symmetric matrix of order $n$, whose $(i,j)^{th}$ element is 1 or 0, according as the vertices $v_i$ and $v_j$ are adjacent or not. Let $d_i = d_G(v_i)$ be the degree of the vertex $v_i$ in $G$ and $D(G)$ be the diagonal matrix with diagonal elements $d_1, d_2. \ldots, d_n$. A graph $G$ is said to be a $t$-regular graph if all the vertices in $G$ have the same degree as $t$. The Laplacian matrix $L(G)$ and signless Laplacian matrix $L_S(G)$ of graph $G$ are defined as $L(G)=D(G)-A(G)$ and $L_S(G)=D(G)+A(G)$, respectively. For any real $\alpha \in [0,1]$, Nikiforov \cite{A_Q-merging_by_Nikiforov} defined the $A_{\alpha}$-matrix of $G$ as,
\begin{equation}\label{eq1}
A_{\alpha}(G) = \alpha D(G) + (1- \alpha) A(G),\qquad \alpha \in [0,1].
\end{equation}
Clearly, when we take $\alpha = 0, \frac{1}{2}$ and $1$, $A_{\alpha}(G)$ becomes $A(G)$, $\frac{1}{2}L_S(G)$ and $D(G)$, respectively. Unless stated otherwise, while mentioning $\alpha$, we understand its domain as $[0,1]$, throughout the paper.

Let $M$ be a square matrix of order $n$. The $M$-characteristic polynomial is defined by 
$\phi_M(\nu) = \det(\nu I_n -M)$, where $I_n$ is the identity matrix of order $n$ and the roots of $\phi_M(\nu)$ are said to be the $M$-eigenvalues. If $M$ is a symmetric matrix associated with a graph $G$, the $M$-eigenvalues of $G$ are real and therefore, we can arrange them in a non-increasing order, denoted by $\lambda_1(M) \geq \lambda_2(M) \geq \ldots \geq \lambda_n(M)$. The multiset of all $M$-eigenvalues of $G$ is called the $M$-spectrum of $G$. Furthermore, if $\lambda_1, \lambda_2, \ldots, \lambda_r$ are all the distinct eigenvalues of $M$ with corresponding multiplicities $m_1, m_2, \ldots, m_r$, then the spectrum of $M$ is denoted by spec$(M) = \{\lambda_1^{[m_1]}, \lambda_2^{[m_2]}, \ldots, \lambda_r^{[m_r]}\}$. If two non isomorphic graphs possess the same $M$-spectrum, they are called $M$-cospectral graphs.

The study of characteristic polynomials and spectra under various graph operations forms a central theme in spectral graph theory, since many such operations allow the spectrum of a complex graph to be related explicitly to the spectra of simpler graphs. Considerable research has been devoted to investigate how the spectral behavior associated with classical matrices, such as the adjacency matrix, Laplacian matrix, and signless Laplacian matrix changes under a wide range of graph operations including the Kronecker product, Cartesian product, corona, classical join, subdivision-vertex join, subdivision-edge join; see \cite{spec_of_corona_by_barik_pati_sarma, spec_of_gr_by_hamm,  spec_of_corona_by_Cui_&_Tian, Q_vertex_join_Q_edge_join_by_arpita, intro_2_th_of_gr_spectra, spectra_of_gr_th&apps,spec_of_rvertjoin_&_redgejoin_by_das_&_panigrahi, spec_of_edge_corona_by_hou_shiu, spec_of_nbd_corona_by_gopalapillai, spec_of_gr_op_based_on_R_gr_by_Lan_Zhou,   spec_of_sub_vert_sub_edge_nbd_corona_by_Liu_Lu,spec_of_subvertexjoin_subedgejoin_by_Liu_&_Zhang, spec_of_nbd_corona_by_Liu_Zhou, spec_of_corona_by_Mcleman_&_Mcnicholas, signless_lap_spec_of_corona_and_edge_corona_by_Wang_Zhou, spectra_by_balamoorthy}. More recently, attention has been grown toward the $A_{\alpha}$-characteristic polynomial of graphs under operations. In $2019$, Li and Wang \cite{A_alph_spec_of_graph_prod_by_Li_Wang} studied the $A_\alpha$-spectra of four types of graph product, namely cartesian product, lexicographic product, directed product and strong product. The $A_\alpha$-characteristic polynomial of corona graphs was computed by Tahir and Zhang \cite{corona_gr_&_alpha_eigen_by_Tahir_&_Zhang} in $2020$. In \cite{1_join_mainak}, Basunia et. al determined the $A_\alpha$-characteristic polynomial of four types of join, namely subdivision-vertex join, subdivision-edge join, $R$-vertex join and $R$-edge join. For additional developments, refer to \cite{A_alph_char_poly_by_oliveira, A_alph_spec_of_joined_union_by_rather, A_alpha_of_join_by_Lin_Liu_Xue, A_alph_spec_of_duplicate_and_corona_by_brondani, A_alph_apec_of_corona_by_najiya, dice_lattice_by_zhang, exploring_energy_bounds_by_basunia, study_on_najiya, alpha_spectra_by_konch, tricyclic_by_basunia}.

An important direction for expanding this line of research arises from the $Q$-$graph$ \cite{intro_2_th_of_gr_spectra} and the $Total$ $graph$ or $T$-$graph$ \cite{criterion_behzad}. The $Q$-graph of $G$, $\mathcal{Q}(G)$ represents the graph that is formed when a new vertex is inserted to each edge of graph $G$, and then two new vertices are made adjacent if they are located on adjacent edges of $G$. Based on this, recently, Das and Panigrahi \cite{Q_vertex_join_Q_edge_join_by_arpita} introduced the concept of $Q$-$vertex$ $join$ operation, represented by $G_1 \dot{\vee}_Q G_2$ and $Q$-$edge$ $join$ operation, represented by $G_1 \underline{\vee}_Q G_2$, for two given graphs $G_1$ and $G_2$. Starting from the disjoint union of $\mathcal{Q}(G_1)$ and $G_2$, the former is obtained by establishing an edge between each vertex of $G_1$ that is in $\mathcal{Q}(G_1)$ and every vertex of $G_2$; whereas we construct the latter by establishing an edge between each newly inserted vertex of $\mathcal{Q}(G_1)$ and every vertex of $G_2$. The Total graph $T(G)$ of graph $G$ is obtained from $G$ by introducing one additional vertex for every edge of $G$, making each such new vertex adjacent to the end vertices of the corresponding edge of $G$, and connecting two newly added vertices, whenever their corresponding edges are incident in $G$. Using this construction, two more graph operations, namely $T$-$vertex$ $join$ $G_1 \dot{\vee}_T G_2$ and $T$-$edge$ $join$ $G_1 \underline{\vee}_T G_2$ were introduced in \cite{spectra_t_vertex_edge_join}. Both are constructed from the disjoint union of $T(G_1)$ and $G_2$, then by making each vertex of $G_1$ in $T(G_1)$ adjacent to every vertex of $G_2$ in the former case, and by connecting each newly added vertex of $T(G_1)$ to every vertex of $G_2$ by an edge in the latter case. Figure \ref{f2} illustrates these four new join operations, with $G_1$ representing a path of length 4 and $G_2$ representing a path of length 3.
\begin{figure}[h!]
\centering
\includegraphics[scale=.75]{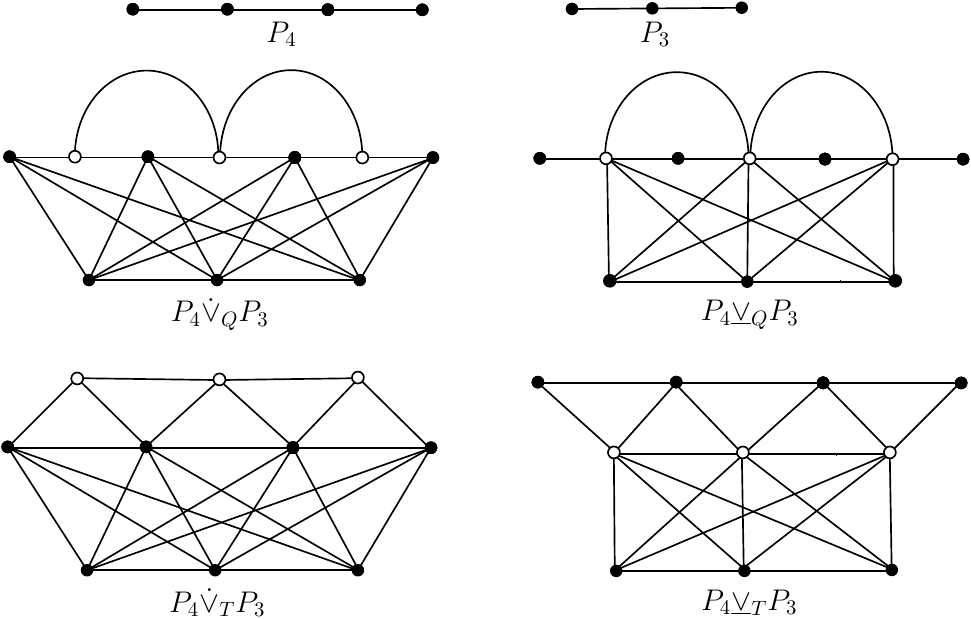}
\caption{$Q$-vertex join, $Q$-edge join, $T$-vertex join and $T$-edge join of $P_4$ and $P_3$}
\label{f2}
\end{figure}

Although the adjacency and Laplacian type spectra of these four join operations have been well investigated in the literature, a complete understanding of their $A_\alpha$-characteristic polynomials and $A_\alpha$-spectra is still lacking. Deriving concise and computable formulas for these spectra is of particular importance, as these results enable efficient spectral computation, facilitate interlacing and perturbation analyses, and support the study of of spectral invariants including the $A_\alpha$-energy and the $A_\alpha$-spectral radius.

Motivated by these considerations, in this paper, we determine the $A_\alpha$-characteristic polynomials and $A_\alpha$-spectra of graphs formed via $Q$-vertex join, $Q$-edge join, $T$-vertex join and $T$-edge join operations, under certain regularity assumptions on the constituent graphs. From an applied standpoint, join-type graph constructions naturally model growth mechanisms in communication networks, molecular compositions and layered infrastructures; see \cite{networks_by_newman, networks_by_easley, mathematical_chemical_by_gutman, structure_by_boccaletti, multilayer_by_kivela}. Since the $A_\alpha$-spectrum plays a crucial role in characterizing diffusion speed, synchronization ability, robustness and epidemic thresholds in networks \cite{A_Q-merging_by_Nikiforov, networks_by_newman, master_stability_by_pecora, epidemic_by_pastor, structure_by_boccaletti}, our results contribute to both the theoretical development and practical applicability of the $A_\alpha$-framework in the analysis of complex networks.

The remaining paper is organized in the following manner. Section \ref{sec3} provides a few preliminary results. In Section \ref{sec4}-\ref{sec7}, we derive the $A_{\alpha}$-characteristic polynomials for $G_1 \dot{\vee}_Q G_2$, $G_1 \underline{\vee}_Q G_2$, $G_1 \dot{\vee}_T G_2$ and $G_1 \underline{\vee}_T G_2$, respectively, where $G_1$ is a regular graph and $G_2$ is an arbitrary graph. In each of these four sections, we also obtain the $A_\alpha$-spectra for the cases where $G_2$ is either regular or complete bipartite. Furthermore, we apply our results to construct new pairs of $A_\alpha$-cospectral graphs. Finally, the concluding remarks are presented in Section \ref{sec8}.


\section{Preliminaries}\label{sec3}
We recall that the incidence matrix of a graph $G$ is the $n \times m$ matrix $R(G)$, whose $(i,j)^{th}$ element is $1$ or $0$, according as the $i^{th}$ vertex is incident to the $j^{th}$ edge or not. The line graph of $G$, represented as $\mathcal{L}(G)$, is the graph with $V(\mathcal{L}(G)) = E(G)$ and there exists an edge between two vertices of $\mathcal{L}(G)$ if and only if the corresponding edges in $G$ share a vertex in common. We find the following well known result in \cite{intro_2_th_of_gr_spectra},
\begin{equation}\label{eq3}
    R(G)^T R(G) = A(\mathcal{L}(G)) + 2I_m.
\end{equation}
For a $t$-regular graph $G$, we have the relation,
\begin{equation}\label{eq4}
    R(G)R(G)^T = A(G) + tI_n.
\end{equation}
Let $P_n$, $C_n$ and $K_n$ denote the path, the cycle and the complete graph on $n$ vertices, respectively. $K_{a,b}$ denotes the complete bipartite graph where $a$ and $b$ are number of vertices in two partite sets of the graph. $\boldsymbol{J_{n\times m}}$ denotes the all $1$'s matrix and $\boldsymbol{0_{n\times m}}$ denotes the all $0$'s matrix of size $n \times m$, while $\boldsymbol{0_n}$ and $\boldsymbol{1_n}$ are used to denote their column counterparts of size $n$. The $M$-$coronal$ $\Upsilon_M(\nu)$ \cite{spec_of_corona_by_Cui_&_Tian, spec_of_corona_by_Mcleman_&_Mcnicholas} for an $n \times n$ matrix $M$ is defined as:
\begin{equation}\label{eq5}
    \Upsilon_M(\nu) = \boldsymbol{1_n^T}(\nu I_n - M)^{-1} \boldsymbol{1_n}.
\end{equation}

Next, we present a few known results that will be crucial in establishing our key findings.
\begin{lemma}\textup{\cite{spec_of_subvertexjoin_subedgejoin_by_Liu_&_Zhang}} \label{lem2}
    For a real number $b$ and a real matrix $M$ of size $n\times n$, we have
    \begin{equation*}\label{eq8}
        \det(M+ bJ_{n\times n}) = \det(M) +  b\boldsymbol{1_n^T} \emph{adj}(M) \boldsymbol{1_n},
    \end{equation*}
    where $\adj(M)$ denotes the adjugate of $M$.
\end{lemma}

\begin{lemma} \textup{\cite{spec_of_subvertexjoin_subedgejoin_by_Liu_&_Zhang}} \label{lem3}
    For a real number $b$ and a real matrix $M$ of size $n\times n$, we have
    \begin{equation*} \label{eq9}
        \det(\nu I_n - M - b J_{n\times n}) = (1-b \Upsilon_M(\nu)) \det (\nu I_n -M).
    \end{equation*}
\end{lemma}

\begin{lemma} \textup{\cite{1_join_mainak}} \label{lem5}
    Given two real numbers $b$ and $c$, let $(bI_n -cJ_{n\times n})$ be an invertible matrix. Then
    \begin{equation*} \label{eq9_1}
        (bI_n -cJ_{n\times n})^{-1} = \frac{1}{b}I_n + \frac{c}{b(b-nc)}J_{n\times n}.
    \end{equation*}
\end{lemma}

\begin{lemma}[Schur complement formula]\textup{\cite{schur_compl_&_its_appl_by_Zhang}} \label{lem1}
    Suppose $A$, $B$, $C$ and $D$ are the matrices of size $p\times p$, $p\times q$, $q\times p$ and $q\times q$, respectively. Then
    \begin{equation*}\label{eq7}
        \begin{split}
            \det \begin{pmatrix}  A & B\\C & D \end{pmatrix} &= \det(D) \det \Big(A - B D^{-1} C\Big), \textrm{when} \enskip D \enskip \textrm{is invertible}.\\ & = \det(A) \det \Big(D -C A^{-1}B \Big), \textrm{when} \enskip A \enskip \textrm{is invertible}.
        \end{split}
    \end{equation*}
\end{lemma}

\begin{lemma} \textup{\cite{intro_2_th_of_gr_spectra}}\label{lem7}
    Suppose $G$ is a graph with order $n$, size $m$ and regularity $t$. Let $\phi_{A(G)}(\nu)$ and $\phi_{A(\mathcal{L}(G))}(\nu)$ be, respectively, the adjacency characteristic polynomials for the graphs $G$ and its line graph $\mathcal{L}(G)$. Then  $\phi_{A(\mathcal{L}(G))}(\nu) = (\nu+2)^{m-n} \phi_{A(G)}(\nu-t+2)$.
\end{lemma}

\begin{lemma}\textup{\cite{spec_of_corona_by_Cui_&_Tian}} \label{lem99.9}
    Suppose $M$ is a $n\times n$ matrix with constant row sum $t$. Then we have $\Upsilon_M(\nu)=\frac{n}{\nu-t}$.
\end{lemma}
 
\begin{lemma} \textup{\cite{A_Q-merging_by_Nikiforov}} \label{lem4}
    The $A_{\alpha}$-spectrum of the graph $K_{a,b}$ is,\\
    \resizebox{.99\textwidth}{!}{${spec}(A_{\alpha}(K_{a,b})) =\bigg\{ \frac{\alpha (a+b)+ \sqrt{{\alpha}^2 (a+b)^2 + 4ab(1-2\alpha)}}{2}, [\alpha a]^{b-1}, [\alpha b]^{a-1}, \frac{\alpha (a+b)- \sqrt{{\alpha}^2 (a+b)^2 + 4ab(1-2\alpha)}}{2} \bigg\}$}.
\end{lemma}

\begin{lemma} \textup{\cite{corona_gr_&_alpha_eigen_by_Tahir_&_Zhang}} \label{lem6}
    The $A_{\alpha}$-coronal of the graph $K_{a,b}$ is,
    \begin{equation*} \label{eq10}
        \Upsilon_{A_{\alpha}(K_{a,b})}(\nu) = \frac{(a+b)\nu
        -\alpha (a+b)^2 + 2ab}{{\nu}^2 -\alpha (a+b)\nu + (2\alpha -1)ab}.
    \end{equation*}
\end{lemma}


\section{$A_{\alpha}$-characteristic polynomial of $Q$-vertex join} \label{sec4}

For the graphs $G_1$ and $G_2$, let $V(G_1)=\{v_1,v_2, \ldots, v_{n_1}\}$ and $V(G_2)=\{u_1,u_2, \ldots, u_{n_2}\}$. Let $I(G_1) = \{v^\prime_1, v^\prime_2, \ldots, v^\prime_{m_1} \}$ be the set of vertices which have been inserted to the edges of $G_1$. Throughout the paper, we index the rows and columns of the matrices associated with $G_1 \dot{\vee}_Q G_2$ and $G_1 \underline{\vee}_Q G_2$ by $V(G_1) \cup I(G_1) \cup V(G_2)$ in order.

\begin{theorem}\label{th1}
    Consider two graphs $G_1$ and $G_2$ with respective orders $n_1$ and $n_2$. Let $m_1$ be the size and $t_1$ be the vertex regularity of $G_1$. Then 
    \begin{align*} \label{eq11}
         \resizebox{.15\textwidth}{!}{$\phi_{A_{\alpha}(G_1 \dot{\vee}_Q G_2)}(\nu)$} &\resizebox{.45\textwidth}{!}{$=(\nu +2 -2\alpha -2\alpha t_1)^{m_1 -n_1} \phi_{A_{\alpha}(G_2)}(\nu- \alpha n_1)$}\\ 
         &\quad \prod_{i=2}^{n_1}\resizebox{.5\textwidth}{!}{$\bigg({\nu}^2 -\Big(t_1 -2 +\lambda_i(A_{\alpha}(G_1)) +\alpha (2+ t_1 + n_2)\Big)\nu$}\\
        &\quad \resizebox{.6\textwidth}{!}{$- 2\alpha n_2(1-\alpha)
        -\Big(1- \alpha(1+t_1+n_2)\Big)\Big(t_1 + \lambda_i(A_{\alpha}(G_1))\Big) \bigg)$}\\
        &\quad \resizebox{.5\textwidth}{!}{$\bigg( (\nu-\alpha t_1 -\alpha n_2)(\nu+2-2\alpha -2t_1) -2t_1(1-\alpha)^2$}\\
        &\quad \resizebox{.5\textwidth}{!}{$-n_1(1-\alpha)^2(\nu+2-2\alpha-2t_1) \Upsilon_{A_{\alpha}(G_2)}(\nu-\alpha n_1) \bigg)$}.
    \end{align*}
\end{theorem}

\begin{proof}
Let us denote the $(0,1)$-incidence matrix of the graph $G_1$ by $R$. Then
\begin{equation}\label{eq12}
    A(G_1 \dot{\vee}_Q G_2) =
    \begin{bmatrix}
        0_{n_1 \times n_1} & R & J_{n_1 \times n_2}\\
        R^T & A(\mathcal{L}(G_1)) & 0_{m_1 \times n_2}\\
        J_{n_2 \times n_1} & 0_{n_2 \times m_1} & A(G_2)
    \end{bmatrix}.
\end{equation}
and
\begin{equation}\label{eq14}
    D(G_1 \dot{\vee}_Q G_2) =
    \begin{bmatrix}
        (t_1 + n_2)I_{n_1} & 0_{n_1 \times m_1} & 0_{n_1 \times n_2}\\
        0_{m_1 \times n_1} & 2t_1I_{m_1} & 0_{m_1 \times n_2}\\
        0_{n_2 \times n_1} & 0_{n_2 \times m_1} & D(G_2) + n_1I_{n_2}
    \end{bmatrix}.
\end{equation}
Therefore using the expressions \eqref{eq12} and \eqref{eq14}, we get 
\begin{equation*}\label{eq15}
    A_{\alpha}(G_1 \dot{\vee}_Q G_2)=
    \begin{bmatrix}
        \alpha(t_1 + n_2)I_{n_1} & (1-\alpha)R & (1-\alpha)J_{n_1 \times n_2}\\
        (1-\alpha)R^T & 2\alpha t_1I_{m_1} +(1-\alpha) A(\mathcal{L}(G_1)) & 0_{m_1 \times n_2}\\
        (1-\alpha)J_{n_2 \times n_1} & 0_{n_2 \times m_1} & A_{\alpha}(G_2) + \alpha n_1 I_{n_2}
    \end{bmatrix}.
\end{equation*}
and so
\begin{align}\label{eq16}
    &\quad \phi_{A_{\alpha}(G_1  \dot{\vee}_Q G_2)}(\nu)\notag\\ 
    &\quad= \det \big(\nu I_{n_1+m_1+n_2}-A_{\alpha}(G_1 \dot{\vee}_Q G_2) \big)\notag\\
    &\quad= \det \left[\begin{smallmatrix}
        (\nu-\alpha(t_1 + n_2)) I_{n_1} & -(1-\alpha)R & -(1-\alpha)J_{n_1 \times n_2}\\
        -(1-\alpha)R^T & (\nu-2\alpha t_1) I_{m_1} -(1-\alpha) A(\mathcal{L}(G_1)) & 0_{m_1 \times n_2}\\
        -(1-\alpha)J_{n_2 \times n_1} & 0_{n_2 \times m_1} & (\nu-\alpha n_1) I_{n_2} - A_{\alpha}(G_2)
            \end{smallmatrix}\right]\notag\\
    &\quad=  \det \big( (\nu-\alpha n_1) I_{n_2} - A_{\alpha}(G_2) \big) \det S \quad (\text{using Lemma } \ref{lem1}),
\end{align}
where
\begin{align*}
    S &=
    \begin{bmatrix}
        \big(\nu-\alpha(t_1 + n_2)\big) I_{n_1} & -(1-\alpha)R\\
        -(1-\alpha)R^T & (\nu-2\alpha t_1 ) I_{m_1} -(1-\alpha)A(\mathcal{L}(G_1))
    \end{bmatrix}\\
    &\quad -
    \begin{bmatrix}
        -(1-\alpha)J_{n_1 \times n_2}\\
        0_{m_1 \times n_2}
    \end{bmatrix}
    \big( (\nu-\alpha n_1) I_{n_2} - A_{\alpha}(G_2) \big)^{-1}
    \begin{bmatrix}
        -(1-\alpha)J_{n_2 \times n_1} & 0_{n_2 \times m_1}
    \end{bmatrix}.
\end{align*}
\begin{align*}
    \det S &= \det \Bigg(
    \begin{bmatrix}
        \big(\nu-\alpha(t_1 + n_2)\big) I_{n_1} & -(1-\alpha)R\\
        -(1-\alpha)R^T & (\nu-2\alpha t_1 ) I_{m_1} -(1-\alpha)A(\mathcal{L}(G_1))
    \end{bmatrix}\\
    &\quad - (1-\alpha)^2
    \begin{bmatrix}
        \Upsilon_{A_{\alpha}(G_2)}(\nu-\alpha n_1) J_{n_1 \times n_1} & 0_{n_1 \times m_1}\\
        0_{m_1 \times n_1} & 0_{m_1 \times m_1}
    \end{bmatrix}
    \Bigg)\\
    &= \det \left[
    \begin{smallmatrix}
        (\nu-\alpha(t_1 + n_2)) I_{n_1} - (1-\alpha)^2\Upsilon_{A_{\alpha}(G_2)}(\nu-\alpha n_1) J_{n_1 \times n_1} & -(1-\alpha)R\\
        -(1-\alpha)R^T & (\nu-2\alpha t_1 ) I_{m_1} -(1-\alpha)A(\mathcal{L}(G_1))
    \end{smallmatrix}\right].
\end{align*}
Applying Lemma \ref{lem1},
\begin{align*}
    \det S &= \det \Big( \big(\nu-\alpha (t_1 +n_2) \big) I_{n_1} -(1-\alpha)^2\Upsilon_{A_{\alpha}(G_2)}(\nu-\alpha n_1) J_{n_1 \times n_1}\Big) \\
    &\quad \det \bigg(  (\nu-2\alpha t_1) I_{m_1} -(1-\alpha)A(\mathcal{L}(G_1)) -(1-\alpha)R^T\Big(\big(\nu-\alpha (t_1 +n_2)\big)I_{n_1} \\
    &\quad - (1-\alpha)^2\Upsilon_{A_{\alpha}(G_2)}(\nu-\alpha n_1)J_{n_1 \times n_1} \Big)^{-1}(1-\alpha)R\bigg)
\end{align*}
Now using Lemma \ref{lem3}, we obtain
\begin{align*}
    &\det \Big( \big(\nu-\alpha (t_1 +n_2) \big) I_{n_1} -(1-\alpha)^2\Upsilon_{A_{\alpha}(G_2)}(\nu-\alpha n_1) J_{n_1 \times n_1}\Big)\\
    &= \big(\nu-\alpha (t_1 +n_2)\big)^{n_1}\Big(1-(1-\alpha)^2 \Upsilon_{A_{\alpha}(G_2)}(\nu-\alpha n_1)\frac{n_1}{\nu-\alpha (t_1 +n_2)} \Big)  =x \quad \text{(say)}
\end{align*}
and using Lemma \ref{lem5}, we get
\begin{align*}
    &\det \bigg(  (\nu-2\alpha t_1) I_{m_1} -(1-\alpha)A(\mathcal{L}(G_1))\\
    &-(1-\alpha)R^T\Big(\big(\nu-\alpha (t_1 +n_2)\big)I_{n_1} - (1-\alpha)^2\Upsilon_{A_{\alpha}(G_2)}(\nu-\alpha n_1)J_{n_1 \times n_1} \Big)^{-1}(1-\alpha)R\bigg)
\end{align*}
\begin{align*}
    &= \det\bigg( (\nu-2\alpha t_1)I_{m_1} -(1-\alpha)A(\mathcal{L}(G_1)) -\frac{(1-\alpha)^2}{\nu-\alpha(t_1 +n_2)} R^TR\\
    &\quad -\frac{4(1-\alpha)^4\Upsilon_{A_{\alpha}(G_2)}(\nu-\alpha n_1)}{\big(\nu-\alpha(t_1 +n_2)\big)\Big(\big(\nu-\alpha(t_1 +n_2)\big) -n_1(1-\alpha)^2 \Upsilon_{A_{\alpha}(G_2)}(\nu-\alpha n_1)\Big)}J_{m_1 \times m_1}\bigg)\\
    &= \det\Big( (\nu-2\alpha t_1)I_{m_1} -(1-\alpha)A(\mathcal{L}(G_1)) -\frac{(1-\alpha)^2}{\nu-\alpha(t_1 +n_2)} R^TR\Big)\\
    &\quad \bigg(1-\Bigg\{\frac{4(1-\alpha)^4\Upsilon_{A_{\alpha}(G_2)}(\nu-\alpha n_1)}{\big(\nu-\alpha(t_1 +n_2)\big)\Big(\big(\nu-\alpha(t_1 +n_2)\big) -n_1(1-\alpha)^2 \Upsilon_{A_{\alpha}(G_2)}(\nu-\alpha n_1)\Big)}\Bigg\} \\
    &\quad \times \Upsilon_{(1-\alpha)A(\mathcal{L}(G_1)) +\frac{(1-\alpha)^2}{\nu-\alpha(t_1 +n_2)} R^TR}(\nu-2\alpha t_1)\bigg), \quad \text{(by Lemma } \ref{lem3}).\\
    &= yz \quad\text{(say),}
\end{align*}
where
\begin{align*}
    y = \det\Big( (\nu-2\alpha t_1)I_{m_1} -(1-\alpha)A(\mathcal{L}(G_1))-\frac{(1-\alpha)^2}{\nu-\alpha(t_1 +n_2)}R^TR\Big)
\end{align*}
and
\begin{align*}
    z &=1-\frac{4(1-\alpha)^4\Upsilon_{A_{\alpha}(G_2)}(\nu-\alpha n_1)}{\big(\nu-\alpha(t_1 +n_2)\big)\Big(\nu-\alpha(t_1 +n_2) -n_1(1-\alpha)^2 \Upsilon_{A_{\alpha}(G_2)}(\nu-\alpha n_1)\Big)}\\
    &\quad\times \Upsilon_{(1-\alpha)A(\mathcal{L}(G_1)) +\frac{(1-\alpha)^2}{\nu-\alpha(t_1 +n_2)} R^TR}(\nu-2\alpha t_1)
\end{align*}
So we have, $\det S = xyz$.\\
Using $\eqref{eq3}$ in the expression of $y$, we get
\begin{align*}
    y &= \det\bigg( (\nu-2\alpha t_1)I_{m_1} -(1-\alpha)A(\mathcal{L}(G_1))-\frac{(1-\alpha)^2}{\nu-\alpha(t_1 +n_2)} \big(A(\mathcal{L}(G_1)) +2I_{m_1}\big)\bigg), \notag\\
    &=\big(\nu-2\alpha t_1+2(1-\alpha)\big)^{m_1 -n_1}\notag\\ 
    &\quad \prod_{i=1}^{n_1}\bigg(\nu -2\alpha t_1 -\frac{2(1-\alpha)^2}{\nu-\alpha(t_1 +n_2)} -\Big( 1-\alpha + \frac{(1-\alpha)^2}{\nu-\alpha(t_1 + n_2)}\Big)\big(\lambda_i(A(G_1))+t_1-2\big)\bigg),\notag\\
    &= \resizebox{.95\textwidth}{!}{$(\nu+2-2\alpha -2\alpha t_1)^{m_1 -n_1} \prod_{i=1}^{n_1}\bigg(\nu -2\alpha t_1 -\frac{2(1-\alpha)^2}{\nu-\alpha(t_1 +n_2)}-\Big( 1 + \frac{1-\alpha}{\nu-\alpha(t_1 + n_2)}\Big)\Big(\lambda_i(A_{\alpha}(G_1)) -2\alpha t_1 -2+t_1+2\alpha\Big)\bigg).$}
\end{align*}
$\mathcal{L}(G_1)$ is a $(2t_1 -2)$ regular graph as $G_1$ is a $t_1$ regular graph. Using \eqref{eq3}, we get that each row sum of the matrix $(1-\alpha)A(\mathcal{L}(G_1)) +\frac{(1-\alpha)^2}{\nu-\alpha(t_1 +n_2)} R^TR$ is $(1-\alpha)(2t_1-2) +\frac{(1-\alpha)^2}{\nu-\alpha(t_1 +n_2)}2t_1$. Therefore from Lemma \ref{lem99.9}, we get
\begin{equation*}
     \Upsilon_{(1-\alpha)A(\mathcal{L}(G_1)) +\frac{(1-\alpha)^2}{\nu-\alpha(t_1 +n_2)} R^TR}(\nu-2\alpha t_1)\\
     = \frac{m_1}{\nu-2\alpha t_1 - (1-\alpha)(2t_1-2) -\frac{(1-\alpha)^2}{\nu-\alpha(t_1 +n_2)}2t_1 }
\end{equation*}
Using this, we get the expression for $z$ as
\begin{align*}
    z &=1-\frac{4(1-\alpha)^4\Upsilon_{A_{\alpha}(G_2)}(\nu-\alpha n_1)}{\nu-\alpha(t_1 +n_2) -n_1(1-\alpha)^2 \Upsilon_{A_{\alpha}(G_2)}(\nu-\alpha n_1)}\\
    &\quad \times \frac{m_1}{(\nu-2\alpha t_1)\big(\nu-\alpha(t_1 +n_2)\big) -(1-\alpha)(2t_1 -2)\big(\nu-\alpha(t_1 +n_2)\big) -2t_1(1-\alpha)^2}
\end{align*}
After simplification, it becomes
\begin{align*}
    \resizebox{.027\textwidth}{!}{  $z$ } &\resizebox{.7\textwidth}{!}{$ =\frac{(\nu-\alpha t_1 -\alpha n_2)}{(\nu-2\alpha t_1)(\nu-\alpha(t_1 +n_2)) -(1-\alpha)(2t_1 -2)(\nu-\alpha(t_1 +n_2)) -2t_1(1-\alpha)^2}$}\\
    &\quad\resizebox{.91\textwidth}{!}{$
    \times \frac{(\nu-\alpha t_1 -\alpha n_2)(\nu-2\alpha t_1 -2(1-\alpha)(t_1 -1))
    -n_1(1-\alpha)^2(\nu-2\alpha t_1 -2(1-\alpha)(t_1 -1)) \Upsilon_{A_{\alpha}(G_2)}(\nu-\alpha n_1) -2t_1 (1-\alpha)^2
    }{\nu-\alpha(t_1 +n_2) -n_1(1-\alpha)^2 \Upsilon_{A_{\alpha}(G_2)}(\nu-\alpha n_1)}$}
\end{align*}
Now using the expressions of $x,y$ and $z$ in $\det S = xyz$, and simplifying again, we get
{\footnotesize
\begin{align}\label{eq100}
    \det S &=(\nu+2-2\alpha -2\alpha t_1)^{m_1 -n_1} \prod_{i=2}^{n_1}\bigg({\nu}^2 -\Big(t_1 -2 +\lambda_i(A_{\alpha}(G_1))+ \alpha(2+t_1 +n_2)\Big)\nu \notag\\
    &\quad -\big(1-\alpha (1+t_1 +n_2)\big)\big(t_1 + \lambda_i(A_{\alpha}(G_1))\big) -2\alpha n_2 (1-\alpha) \bigg)\notag\\
    &\quad \bigg( (\nu-\alpha t_1 -\alpha n_2)(\nu+2-2\alpha -2t_1) \notag\\
    &\quad-n_1(1-\alpha)^2(\nu+2-2\alpha -2t_1)\Upsilon_{A_{\alpha}(G_2)}(\nu-\alpha n_1) -2t_1 (1-\alpha)^2\bigg)
\end{align}
}
\eqref{eq100} together with \eqref{eq16} gives the required result.
\end{proof}

Assuming that $G_2$ is a $t_2$-regular graph, the next corollary computes the $A_{\alpha}$-spectrum of $G_1 \dot{\vee}_Q G_2$. 

\begin{corollary}\label{cor1.1}
    Consider two graphs $G_1$ and $G_2$ with respective orders $n_1$ and $n_2$, and regularities $t_1$ and $t_2$. Further, the size of $G_1$ is $m_1$. Then
    \begin{enumerate}[label={\upshape (\alph*)}]
        \item When $t_1 =1$, the $A_{\alpha}$-spectrum of $G_1 \dot{\vee}_Q G_2$ comprises:
        \begin{enumerate}[label= {\upshape (\roman*)}]
            \item one simple eigenvalue $\alpha (1+n_2)$;
            \item the eigenvalues $2\alpha + \lambda_j( A_{\alpha}(G_2))$ for $j=2,3, \ldots, n_2$ and
            \item three eigenvalues obtained by solving
            \begin{align*}
                \resizebox{0.87\textwidth}{!}{$(\nu-2\alpha -t_2)(\nu-\alpha -\alpha n_2)(\nu-2\alpha)-2n_2(1-\alpha)^2(\nu-2\alpha) - 2(1-\alpha)^2(\nu-2\alpha -t_2)=0$}
            \end{align*}
        \end{enumerate}

        \item When $t_1 \geq 2$, the $A_{\alpha}$-spectrum of $G_1 \dot{\vee}_Q G_2$ comprises:
        \begin{enumerate}[label= {\upshape (\roman*)}]
            \item one eigenvalue $2\alpha(1+t_1) -2$ with multiplicity $m_1 -n_1$;
            \item the eigenvalues $\alpha n_1 + \lambda_j(A_{\alpha}(G_2))$ for $j=2,3, \ldots, n_2$;
            \item eigenvalues obtained by solving the equations,
            \begin{align*}
                \nu^2 &-\big(t_1 -2 +\lambda_i(A_{\alpha}(G_1))+\alpha(2+t_1 +n_2)\big)\nu-2\alpha (1-\alpha)n_2\\
                & -\big(1-\alpha(1+t_1 +n_2)\big) \big(t_1 +\lambda_i(A_{\alpha}(G_1))\big) = 0, \enskip i= 2,3,\ldots,n_1 \enskip \text{and}
            \end{align*}
            \item three eigenvalues obtained by solving the equation,
            \begin{align*}
                &(\nu-\alpha n_1 -t_2)(\nu-\alpha t_1 -\alpha n_2)(\nu+2-2\alpha -2t_1) \\
                &-n_1 n_2 (1-\alpha)^2(\nu+2-2\alpha -2t_1) -2(1-\alpha)^2 t_1(\nu-\alpha n_1 -t_2)=0.
            \end{align*}
        \end{enumerate}
    \end{enumerate}
\end{corollary}

\begin{proof}
    The matrix $A_{\alpha}(G_2)$ has constant row sum $t_2$, because $G_2$ is a $t_2$-regular graph of order $n_2$. This implies $\Upsilon_{A_{\alpha}(G_2)}(\nu-\alpha n_1) = \frac{n_2}{\nu - \alpha n_1 -t_2}$ due to Lemma \ref{lem99.9}. Further, since the matrix $A(G_2)$ possesses the eigenvalue $t_2$, 
    $A_{\alpha}(G_2) = \alpha t_2 I_{n_2} + (1-\alpha)A(G_2)$ too has the same eigenvalue. Applying Theorem \ref{th1}, these give us
    
    {\footnotesize
        \begin{align} \label{eq28}
            \phi_{A_{\alpha}(G_1 \dot{\vee}_Q G_2)}(\nu) &= (\nu+2-2\alpha -2\alpha t_1)^{m_1-n_1} \prod_{j=2}^{n_2}\Big(\nu-\alpha n_1 -\lambda_j(A_{\alpha}(G_2))\Big)\notag\\
            &\quad \prod_{i=2}^{n_1}\bigg(  {\nu}^2 -\Big(t_1 -2 +\lambda_i(A_{\alpha}(G_1)) +\alpha (2+ t_1 + n_2)\Big)\nu\notag\\
            &\quad- 2\alpha n_2(1-\alpha)
            -\Big((1- \alpha(1+t_1+n_2)\Big)\Big(t_1 + \lambda_i(A_{\alpha}(G_1))\Big) \bigg)\notag\\
            &\quad \bigg( (\nu-\alpha n_1 -t_2)(\nu-\alpha t_1 -\alpha n_2)(\nu+2-2\alpha -2t_1)\notag\\
            &\quad-2t_1(1-\alpha)^2(\nu-\alpha n_1 -t_2)
            -n_1 n_2(1-\alpha)^2(\nu+2-2\alpha-2t_1) \bigg)
        \end{align}
    }
    \begin{enumerate}[label={\upshape (\alph*)}]
        \item When $t_1 =1$, we observe that $G_1 =  P_2$, $n_1 =2$ and $m_1 =1$. The $A_{\alpha}$-spectra of $P_2$ is $\{1,2\alpha -1\}$. Using the information, the result follows from \eqref{eq28}.

        \item When $t_1 \geq 2$, $G_1$ does not possess any pendant vertex, so $G_1$ can not be a tree. This implies $m_1 \geq n_1$. From \eqref{eq28}, we get the required result.\qedhere
    \end{enumerate}
\end{proof}

Before moving forward, it is worth validating the practical usefulness of Corollary \ref{cor1.1}. The following example demonstrates how the corollary enables a direct computation of the $A_\alpha$-spectrum of a $Q$-vertex join graph from the spectral data of its factor graphs, for a particular choice of $\alpha$, and confirms the result via an independent SageMath calculation.

\begin{example}
    Let $H$ be the graph depicted in Figure \ref{f1}. We aim to determine the complete $A_{\frac{1}{2}}$-spectrum of $H$. Due to the relatively complex structure of $H$, constructing the matrix $A_{\alpha}(H)$ (with $\alpha = \frac{1}{2}$), and computing its eigenvalues directly is quite tedious, even with the help of SageMath. Nevertheless, a numerical computation in SageMath yields $\text{spec}_{\text{SageMath}}(A_{\frac{1}{2}}(H))$ = $\{ 5.632, 3.790, [3.5]^3, [2]^6, 1.077\}$. Now the key observation is that the graph $H$ can be expressed in a much simpler form as the $Q$-vertex join of two very well-known graphs. Indeed, setting $G_1=K_4$ and $G_2=P_2$, we have $H= G_1 \dot{\vee}_Q G_2$, with $t_1=3, n_1=4, m_1=6$ and $t_2=1, n_2=2$. Moreover, it is easy to check that $\text{spec}({A_{\frac{1}{2}}(G_1)}) = \{ 3, [1]^3 \}$ and $\text{spec}({A_{\frac{1}{2}}(G_2)}) = \{ 1, 0 \}$. Substituting these data into the formula provided in Corollary \ref{cor1.1} (b), we find that the theoretical spectrum $\text{spec}_{\text{Theory}}(A_{\frac{1}{2}}(H)) = \text{spec}_{\text{Theory}}(A_{\frac{1}{2}}(G_1 \dot{\vee}_Q G_2))$ consists of the following eigenvalues: $2$ with multiplicity two, $2$ with multiplicity one, $3.5$ and $2$, each with multiplicity three, and $5.632, 3.790$ and $1.077$, each with multiplicity one. 
    \par Thus, both approaches yield exactly the same $A_{\frac{1}{2}}$-spectrum, confirming the correctness and practical usefulness of Corollary \ref{cor1.1} for this graph.
    \begin{figure}[h!]
    \centering
    \includegraphics[scale=.65]{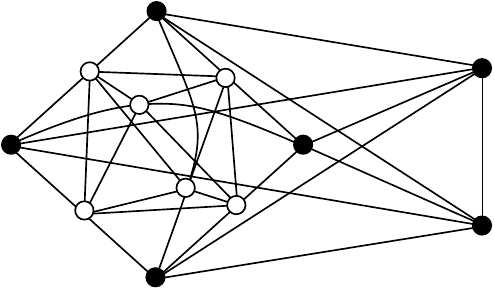}
    \caption{The graph $H$}
    \label{f1}
    \end{figure}
\end{example}

Assuming that $G_2$ is $K_{a,b}$ , the next corollary computes the $A_{\alpha}$-spectrum of $G_1 \dot{\vee}_Q G_2$. 

\begin{corollary}\label{cor1.2}
    Consider a graph $G_1$ of order $n_1$, size $m_1$ and regularity $t_1$. Let $G_2=K_{a,b}$ be the complete bipartite graph of order $n_2 = a+b$. Then
   \begin{enumerate}[label={\upshape (\alph*)}]
        \item When $t_1 =1$, the $A_{\alpha}$-spectrum of $G_1 \dot{\vee}_Q G_2$ comprises:
        \begin{enumerate}[label= {\upshape (\roman*)}]
            \item one simple eigenvalue $\alpha (1+n_2)$;
            \item one eigenvalue $\alpha (a+2)$ with multiplicity $b-1$;
            \item one eigenvalue $\alpha (b+2)$ with multiplicity $a-1$ and
            \item four eigenvalues obtained by solving the equation,
            \begin{align*}
                &\Big( {\nu}^2 -\alpha(3+n_2)\nu + 2\alpha^2 n_2 +4\alpha -2 \Big) \Big( {\nu}^2 -\alpha (4 + n_2)\nu + 4{\alpha}^2 + 2{\alpha}^2n_2 \\
                & +2\alpha ab -ab \Big)-2(1-\alpha)^2(\nu-2\alpha) \Big( (\nu-2\alpha)n_2 -\alpha{n_2}^2 + 2ab \Big)=0.
            \end{align*}
        \end{enumerate}

        \item When $t_1 \geq 2$, the $A_{\alpha}$-spectrum of $G_1 \dot{\vee}_Q G_2$ comprises:
        \begin{enumerate}[label= {\upshape (\roman*)}]
            \item one eigenvalue $2\alpha(1+t_1)-2$ with multiplicity $m_1 -n_1$;
            \item one eigenvalue $\alpha (n_1 +a)$ with multiplicity $b-1$;
            \item one eigenvalue $\alpha (n_1 +b)$ with multiplicity $a-1$;
            \item eigenvalues obtained by solving the equations,
            \begin{align*}
                &{\nu}^2 - \big(t_1 -2+ \lambda_i(A_{\alpha}(G_1)) + \alpha(2+t_1 +n_2)\big)\nu- 2\alpha(1-\alpha)n_2 \\
                & -\big(1-\alpha(1+t_1 +n_2) \big)\big(t_1 +\lambda_i(A_{\alpha}(G_1)) \big)=0 \enskip \text{for } i=2,3,\ldots,n_1 \enskip \text{and} 
            \end{align*}
            \item four eigenvalues obtained by solving the equation,
            \begin{align*}
                &\Big( {\nu}^2 -(\alpha t_1 +\alpha n_2 +2\alpha +2t_1 -2)\nu \\
                &+(2\alpha^2 n_2 + 2\alpha t_1^2 +2\alpha t_1 n_2 +2\alpha t_1 -2\alpha n_2 -2t_1)   \Big) \\
                &\Big( {\nu}^2 -\alpha( 2n_1 + n_2)\nu +( {\alpha}^2n_1^2 + {\alpha}^2n_1n_2 +2\alpha ab -ab)  \Big)\\
                &-n_1(1-\alpha)^2(\nu+2-2\alpha -2t_1)\big((\nu-\alpha n_1)n_2 -\alpha n_2^2 +2ab \big)=0.
            \end{align*}
        \end{enumerate}
    \end{enumerate}
\end{corollary}

\begin{proof}
    From Lemmas \ref{lem4} and \ref{lem6}, we get\\
    $\text{spec}(A_{\alpha}(G_2)) = \bigg\{ \frac{\alpha (a+b)+ \sqrt{{\alpha}^2 (a+b)^2 + 4ab(1-2\alpha)}}{2}$, $[\alpha a]^{b-1}$,$ [\alpha b]^{a-1}$, $\frac{\alpha (a+b)- \sqrt{{\alpha}^2 (a+b)^2 + 4ab(1-2\alpha)}}{2} \bigg\}$
    \\
    and
    \begin{equation*} \label{eq32}
        \Upsilon_{A_{\alpha}(G_2)}(\nu) = \frac{(a+b)\nu -\alpha (a+b)^2 + 2ab}{{\nu}^2 -\alpha (a+b)\nu + (2\alpha -1)ab}.
    \end{equation*}
    We use these in Theorem \ref{th1} to get
    \begingroup
    \allowdisplaybreaks
   {\footnotesize
   \begin{align*} 
        &\phi_{A_{\alpha}(G_1 \dot{\vee}_Q G_2)}(\nu)\\
        &= \big(\nu+2-2\alpha -2\alpha t_1\big)^{m_1 -n_1} \Bigg(\nu-\alpha n_1 - \frac{\alpha(a+b) + \sqrt{\alpha^2 {(a+b)}^2 + 4ab(1-2\alpha)}}{2}\Bigg)\notag\\
        &\quad\big(\nu-\alpha n_1 -\alpha a\big)^{b-1}\big(\nu-\alpha n_1 -\alpha b\big)^{a-1} \Bigg(\nu-\alpha n_1 - \frac{\alpha(a+b) - \sqrt{\alpha^2 {(a+b)}^2 + 4ab(1-2\alpha)}}{2}\Bigg) \notag\\
        &\quad \bigg\{\prod_{i=2}^{n_1}\bigg(  {\nu}^2 -\Big(t_1 -2 +\lambda_i(A_{\alpha}(G_1)) +\alpha (2+ t_1 + n_2)\Big)\nu- 2\alpha n_2(1-\alpha)\notag\\
        &\quad
        -\Big((1- \alpha(1+t_1+n_2)\Big)\Big(t_1 + \lambda_i(A_{\alpha}(G_1))\Big) \bigg)\bigg\} \bigg( (\nu-\alpha t_1 -\alpha n_2)(\nu+2-2\alpha -2t_1) \notag\\
        &\quad  -2t_1(1-\alpha)^2 -n_1\big(1-\alpha\big)^2\big(\nu+2-2\alpha -2t_1\big) \frac{(a+b)(\nu-\alpha n_1) -\alpha (a+b)^2 + 2ab}{(\nu-\alpha n_1)^2 -\alpha (a+b)(\nu-\alpha n_1) + (2\alpha -1)ab} \bigg).
    \end{align*}}
    \endgroup

    We observe that $\alpha n_1 + \frac{\alpha(a+b) + \sqrt{\alpha^2 {(a+b)}^2 + 4ab(1-2\alpha)}}{2}$ and $\alpha n_1 + \frac{\alpha(a+b) - \sqrt{\alpha^2 {(a+b)}^2 + 4ab(1-2\alpha)}}{2}$ are the zeros of the polynomial, $(\nu-\alpha n_1)^2 -\alpha (a+b)(\nu-\alpha n_1) + (2\alpha -1)ab$.
    
    Therefore, we have the characteristic polynomial as 
    \begin{align}\label{eq35}
        \phi_{A_{\alpha}(G_1 \dot{\vee}_Q G_2)}(\nu)&= \big(\nu+2-2\alpha -2\alpha t_1\big)^{m_1 -n_1} \big(\nu-\alpha n_1 -\alpha a\big)^{b-1}  \big(\nu-\alpha n_1 -\alpha b\big)^{a-1}\notag\\
        &\quad \prod_{i=2}^{n_1}\bigg({\nu}^2 -\Big(t_1 -2 +\lambda_i(A_{\alpha}(G_1)) +\alpha (2+ t_1 + n_2)\Big)\nu\notag\\
        &\quad- 2\alpha n_2(1-\alpha)
        -\Big((1- \alpha(1+t_1+n_2)\Big)\Big(t_1 + \lambda_i(A_{\alpha}(G_1))\Big) \bigg)\notag\\
        &\quad \bigg(\Big({\nu}^2 -(\alpha t_1 +\alpha n_2 +2\alpha +2t_1 -2)\nu\notag\\
        &\quad+(2\alpha^2 n_2 + 2\alpha t_1^2 +2\alpha t_1 n_2 +2\alpha t_1 -2\alpha n_2 -2t_1)   \Big) \notag\\
        &\quad \Big({\nu}^2 -\alpha( 2n_1 + n_2)\nu +( {\alpha}^2n_1^2 + {\alpha}^2n_1n_2 +2\alpha ab -ab)  \Big)\notag\\
        &\quad-n_1(1-\alpha)^2(\nu +2-2\alpha -2t_1)\Big((\nu-\alpha n_1)n_2 -\alpha n_2^2 +2ab \Big)\bigg).
    \end{align}
    \begin{enumerate}[label={\upshape (\alph*)}]
        \item When $t_1 =1$, we observe that $G_1 =  P_2$, $n_1 =2$ and $m_1 =1$. The $A_{\alpha}$-spectra of $P_2$ is $\{1,2\alpha -1\}$. Using the information, the result follows from \eqref{eq35}.

        \item When $t_1 \geq 2$, $G_1$ can not be a tree, which implies $m_1 \geq n_1$. Consequently, the desired result is obtained from \eqref{eq35}.\qedhere
    \end{enumerate}
\end{proof}

The following example serves as a quick verification of Corollary \ref{cor1.2}.

\begin{example}
    Consider the graph $\tilde{H}$ in Figure \ref{f3}. We compute the complete $A_{\frac{1}{2}}$-spectrum of the graph in SageMath, which is $\text{spec}_{\text{SageMath}}(A_{\frac{1}{2}}(\tilde{H}))$ = $\{ 6.336, 4.681, [4]^3$, \linebreak $ [3]^2, [2.5]^3, [2]^3, 1.484\}$. It can be observed that $\tilde{H}= G_1 \dot{\vee}_Q G_2$, where $G_1=K_4$ and $G_2=K_{2,2}$. Thus we have $t_1=3, n_1=4, m_1=6$ and $n_2=a+b=2+2=4$. Moreover $\text{spec}({A_{\frac{1}{2}}(G_1)}) = \{ 3, [1]^3 \}$ and $\text{spec}({A_{\frac{1}{2}}(G_2)}) = \{ 2, [1]^2, 0 \}$. Using these we find that the $A_{\frac{1}{2}}$-spectrum of $H$ provided by Corollary \ref{cor1.2} (b), $\text{spec}_{\text{Theory}}(A_{\frac{1}{2}}(\tilde{H}))$ consists of the following eigenvalues: $2$ with multiplicity two, $3$ with multiplicity one, $3$ with multiplicity one, $4$ and $2.5$, each with multiplicity three, and $6.336, 4.681, 2$ and $1.484$, each with multiplicity one. Thus both $A_{\frac{1}{2}}$-spectrum of $\tilde{H}$ comes out to be exactly the same, validating the correctness of Corollary \ref{cor1.2}.
    \begin{figure}[h!]
    \centering
    \includegraphics[scale=.55]{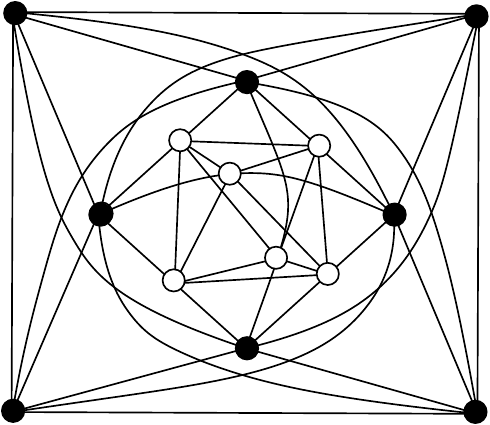}
    \caption{The graph $\tilde{H}$}
    \label{f3}
    \end{figure}
\end{example}

The subsequent corollary shows how to generate new pairs of graphs that are $A_\alpha$-cospectral, from an existing pair of such graphs.

\begin{corollary}\label{cor1.3}
    \begin{enumerate}[label={\upshape (\alph*)}]
        \item Consider two regular graphs $G_1$ and $G_2$, and an arbitrary graph $H$. Then $A_{\alpha}$-cospectralness of $G_1$ and $G_2$ implies that $G_1\dot{\vee}_Q H$ and $G_2 \dot{\vee}_Q H$ are $A_\alpha$-cospectral.
        \item Consider a regular graph $G$. Suppose $H_1$ and $H_2$ are two graphs satisfying the condition $\Upsilon_{A_{\alpha}(H_1)}(\nu) = \Upsilon_{A_{\alpha}(H_2)}(\nu)$. Then $A_{\alpha}$-cospectralness of $H_1$ and $H_2$ implies that $G\dot{\vee}_Q H_1$ and $G \dot{\vee}_Q H_2$ are $A_\alpha$-cospectral.
    \end{enumerate}
\end{corollary}

\begin{proof}
    Two $A_{\alpha}$-cospectral regular graphs possess the exact same order, the same size, as well as the same regularity. We make use of this information and apply Theorem \ref{th1} to the graphs concerned. The results follow just by comparing their $A_{\alpha}$-characteristic polynomials.
\end{proof}

We provide an application of Corollary \ref{cor1.3} through an example below that demonstrates the construction of new $L_S$-cospectral graphs.

\begin{example}
    We take two regular graphs $G_1$ and $G_2$, as illustrated in Figure \ref{f4}. Since they are adjacency cospectral regular graphs (see \cite{which_graphs_by_vandam}), they are $L_S$-cospectral as well. Further, take $H$ as $P_2$. Applying Corollary \ref{cor1.3}(a) to these graphs with $\alpha = \frac{1}{2}$, and noting that for any graph $G$, $L_S(G) = 2A_{\frac{1}{2}}(G)$, we conclude that the graphs $G_1\dot{\vee}_Q H$ and $G_2\dot{\vee}_Q H$, shown in Figure \ref{f5}, are also $L_S$-cospectral.
\end{example}


\section{$A_{\alpha}$-characteristic polynomial of $Q$-edge join} \label{sec5}

\begin{theorem}\label{th2}
    Consider two graphs $G_1$ and $G_2$ with respective orders $n_1$ and $n_2$. Let $m_1$ be the size and $t_1$ be the vertex regularity of $G_1$. Then
    
    {\footnotesize
        \begin{align*}
            \phi_{A_{\alpha}(G_1 \underline{\vee}_Q G_2)}(\nu) &=\big(\nu+2-\alpha (2t_1 +n_2 +2)\big)^{m_1 -n_1} \phi_{A_{\alpha}(G_2)}(\nu- \alpha m_1)\\
            &\quad \prod_{i=2}^{n_1}\bigg({\nu}^2 -\Big( \alpha(t_1 +n_2 +2)+t_1 -2 +\lambda_i(A_{\alpha}(G_1))\Big)\nu \\
            &\quad+\alpha^2 t_1 n_2 + \big( \alpha(t_1 +1)-1\big)\big( t_1 +\lambda_i(A_{\alpha}(G_1))\big)\bigg)\\
            &\quad \bigg((\nu-\alpha t_1)(\nu-2\alpha t_1 -\alpha n_2) -(1-\alpha)(2t_1 -2)(\nu-\alpha t_1 +1 -\alpha)\\ 
            &\quad-2(1-\alpha)^2 -m_1(1-\alpha)^2(\nu-\alpha t_1)\Upsilon_{A_{\alpha}(G_2)}(\nu-\alpha m_1)   \bigg)
        \end{align*}
        }
\end{theorem}

\begin{figure}[h!]
    \centering
    \includegraphics[scale=.5]{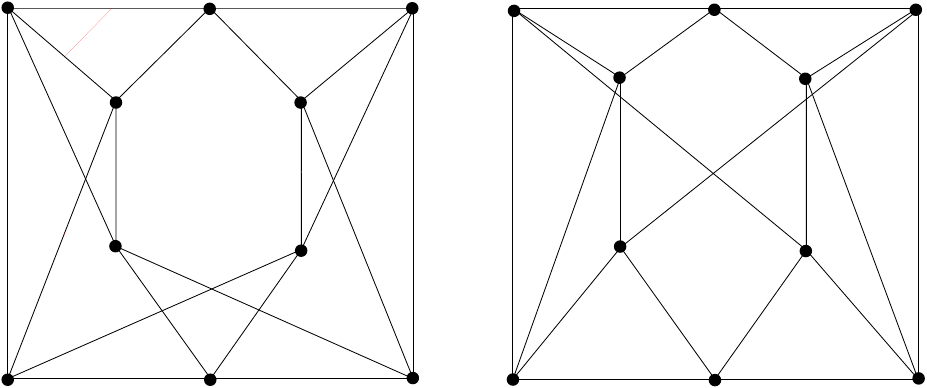}
    \caption{Two $L_S$-cospectral graphs $G_1$ and $G_2$}
    \label{f4}
    \end{figure}
    \begin{figure}[h!]
    \centering
    \includegraphics[scale=.6]{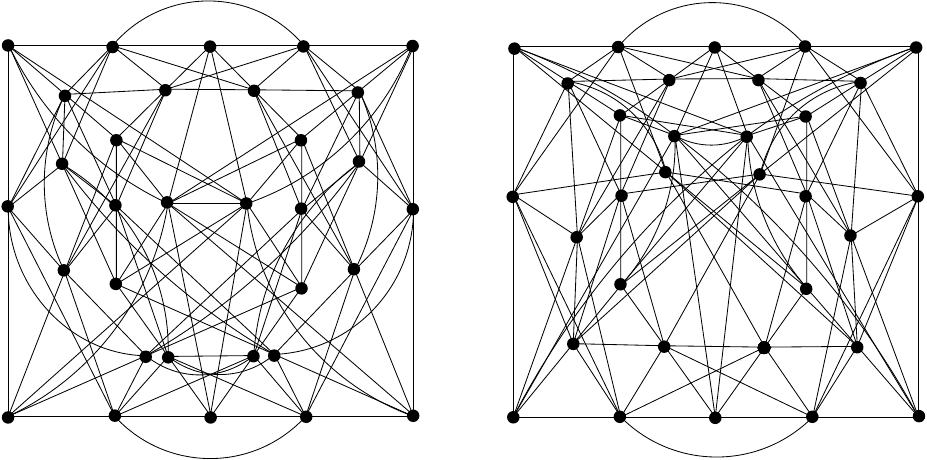}
    \caption{Two new $L_S$-cospectral graphs $G_1\dot{\vee}_Q P_2$ and $G_2\dot{\vee}_Q P_2$}
    \label{f5}
\end{figure}

\begin{proof}
We have
\begin{equation*}\label{eq38}
    A(G_1 \underline{\vee}_Q G_2) =
    \begin{bmatrix}
        0_{n_1 \times n_1} & R & 0_{n_1 \times n_2}\\
        R^T & A(\mathcal{L}(G_1)) & J_{m_1 \times n_2}\\
        0_{n_2 \times n_1} & J_{n_2 \times m_1} & A(G_2)
    \end{bmatrix}
\end{equation*}
and
\begin{equation*}\label{eq40}
    D(G_1 \underline{\vee} G_2) =
    \begin{bmatrix}
        t_1 I_{n_1} & 0_{n_1 \times m_1} & 0_{n_1 \times n_2}\\
        0_{m_ \times n_1} & (2t_1+n_2) I_{m_1} & 0_{m_1 \times n_2}\\
        0_{n_2 \times n_1} & 0_{n_2 \times m_1} & D(G_2) + m_1 I_{n_2}
    \end{bmatrix}.
\end{equation*}
Therefore
\begin{equation*}\label{eq41}
    A_{\alpha}(G_1 \underline{\vee}_Q G_2)=
    \begin{bmatrix}
        \alpha t_1 I_{n_1} & (1-\alpha)R & 0_{n_1 \times n_2}\\
        (1-\alpha)R^T & \alpha(2t_1+n_2)I_{m_1} +(1-\alpha)A(\mathcal{L}(G_1)) & (1-\alpha)J_{m_1 \times n_2}\\
        0_{n_2 \times n_1} & (1-\alpha)J_{n_2 \times m_1} & A_{\alpha}(G_2) + \alpha m_1 I_{n_2}
    \end{bmatrix}.
\end{equation*}
\begin{align}\label{eq42}
    \phi_{A_{\alpha}(G_1 \underline{\vee}_Q G_2)}(\nu)&= \det \big(\nu I_{n_1+m_1+n_2}-A_{\alpha}(G_1 \underline{\vee}_Q G_2) \big)\notag\\
    &=\resizebox{.75\textwidth}{!}{$
         \det 
         \begin{bmatrix}
              (\nu-\alpha t_1) I_{n_1} & -(1-\alpha)R & 0_{n_1 \times n_2}\\
              -(1-\alpha)R^T & (\nu-2\alpha t_1-\alpha n_2 ) I_{m_1}-(1-\alpha)A(\mathcal{L}(G_1)) & -(1-\alpha)J_{m_1 \times n_2}\\
              0_{n_2 \times n_1} & -(1-\alpha)J_{n_2 \times m_1} & (\nu-\alpha m_1) I_{n_2} - A_{\alpha}(G_2)
         \end{bmatrix}
                              $}\notag\\
    &=  \det \big( (\nu-\alpha m_1) I_{n_2} - A_{\alpha}(G_2) \big) \det S \quad (\text{by Lemma } \ref{lem1}),
\end{align}
where
\begingroup
\allowdisplaybreaks
\begin{align*}
    S &=
    \begin{bmatrix}
        (\nu-\alpha t_1) I_{n_1} & -(1-\alpha)R\\
        -(1-\alpha)R^T & (\nu-2\alpha t_1-\alpha n_2) I_{m_1} -(1-\alpha)A(\mathcal{L}(G_1))
    \end{bmatrix}\\
    &\quad-
    \begin{bmatrix}
        0_{n_1 \times n_2}\\
        -(1-\alpha)J_{m_1 \times n_2}
    \end{bmatrix}
    \Big( (\nu-\alpha m_1) I_{n_2} - A_{\alpha}(G_2) \Big)^{-1}
    \begin{bmatrix}
        0_{n_2 \times n_1} & -(1-\alpha)J_{n_2 \times m_1}
    \end{bmatrix}\\    
    &= \resizebox{.85\linewidth}{!}{${
    \begin{bmatrix}
        (\nu-\alpha t_1) I_{n_1} & -(1-\alpha)R \\
        -(1-\alpha)R^T & (\nu-2\alpha t_1 -\alpha n_2 ) I_{m_1} - (1-\alpha)A(\mathcal{L}(G_1))-(1-{\alpha})^2\Upsilon_{A_{\alpha}(G_2)}\big(\nu-\alpha m_1\big)J_{m_1 \times m_1}
    \end{bmatrix}
    }$}.   
\end{align*}
\endgroup

Applying Lemma \ref{lem1},
\begin{align}\label{eq3000}
    \det S &= \det \big( (\nu-\alpha t_1) I_{n_1}\big) \det \Big( (\nu-2\alpha t_1-\alpha n_2)I_{m_1}-(1-\alpha)A(\mathcal{L}(G_1)) \notag\\
    &\quad- (1-\alpha)^2\Upsilon_{A_{\alpha}(G_2)}(\nu-\alpha m_1)J_{m_1 \times m_1} -\frac{(1-\alpha)^2}{(\nu-\alpha t_1)}R^TR \Big)\notag
\end{align}
Using Lemma $\ref{lem3}$, we get the right hand side expression as
\begin{align}
    & (\nu-\alpha t_1)^{n_1} \det \Big(  (\nu-2\alpha t_1 - \alpha n_2) I_{m_1} -(1-\alpha)A(\mathcal{L}({G_1})) - \frac{(1-\alpha)^2}{(\nu-\alpha t_1)}R^TR \Big)\notag\\
    & \Big( 1-  (1-\alpha)^2\Upsilon_{A_{\alpha}(G_2)}(\nu-\alpha m_1) \Upsilon_{(1-\alpha)A(\mathcal{L}(G_1))+\frac{(1-\alpha)^2}{(\nu-\alpha t_1)}R^TR}(\nu-2\alpha t_1-\alpha n_2)  \Big)\notag
\end{align}

{\footnotesize
\begin{align}
    &= (\nu-\alpha t_1)^{n_1} \det \bigg( (\nu-2\alpha t_1 - \alpha n_2) I_{m_1} -\Big\{\frac{2(1-\alpha)^2}{(\nu-\alpha t_1)} I_{m_1} +\Big(1-\alpha +\frac{(1-\alpha)^2}{(\nu-\alpha t_1)} \Big)A(\mathcal{L}(G_1))\Big\} \bigg)\notag\\
    &\quad \Big( 1-  (1-\alpha)^2\Upsilon_{A_{\alpha}(G_2)}(\nu-\alpha m_1) \Upsilon_{(1-\alpha)A(\mathcal{L}(G_1))+\frac{(1-\alpha)^2}{(\nu-\alpha t_1)}R^TR}(\nu-2\alpha t_1-\alpha n_2)  \Big),\notag\\
    &=(\nu-\alpha t_1)^{n_1} xy, 
\end{align}
}
where 
\begin{equation*}
    x=\det \bigg( (\nu-2\alpha t_1 - \alpha n_2) I_{m_1} -\Big\{\frac{2(1-\alpha)^2}{(\nu-\alpha t_1)} I_{m_1} +\Big(1-\alpha +\frac{(1-\alpha)^2}{(\nu-\alpha t_1)} \Big)A(\mathcal{L}(G_1))\Big\} \bigg)
\end{equation*}
and 
\begin{equation*}
    y= \Big( 1-  (1-\alpha)^2\Upsilon_{A_{\alpha}(G_2)}(\nu-\alpha m_1) \Upsilon_{(1-\alpha)A(\mathcal{L}(G_1))+\frac{(1-\alpha)^2}{(\nu-\alpha t_1)}R^TR}(\nu-2\alpha t_1-\alpha n_2)  \Big).
\end{equation*}
Using Lemma \ref{lem7}, we get
\begingroup
\allowdisplaybreaks
\begin{align*}
    x &= \big(\nu-2\alpha t_1 -\alpha n_2 +2(1-\alpha)\big)^{m_1 -n_1}\\
    &\quad \prod_{i=1}^{n_1}\bigg(\nu-2\alpha t_1 -\alpha n_2 -\frac{2(1-\alpha)^2}{(\nu-\alpha t_1)} - \Big(1-\alpha+\frac{(1-\alpha)^2}{(\nu-\alpha t_1)} \Big)\big(\lambda_i(A(G_1)) +t_1 -2 \big) \bigg)\\
    &= \big(\nu+2-\alpha ( 2t_1 + n_2 +2)\big)^{m_1 -n_1}\prod_{i=1}^{n_1}\bigg(\nu-2\alpha t_1 -\alpha n_2 -\frac{2(1-\alpha)^2}{(\nu-\alpha t_1)}\\
    &\quad- \Big(1+\frac{1-\alpha}{\nu-\alpha t_1} \Big)\Big(\lambda_i(A_{\alpha}(G_1)) -2\alpha t_1 +t_1 -2 +2\alpha\Big) \bigg)\\
    &= \frac{\big(\nu+2-\alpha ( 2t_1 + n_2 +2)\big)^{m_1 -n_1}}{(\nu-\alpha t_1)^{n_1}}\\
    &\quad \Big((\nu-\alpha t_1)(\nu-2\alpha t_1 -\alpha n_2) -2(1-\alpha)^2 -2(\nu-\alpha t_1 +1 -\alpha)(1-\alpha)(t_1 -1)\Big)\\
    &\quad \prod_{i=2}^{n_1}\Big((\nu-\alpha t_1)(\nu-2\alpha t_1 -\alpha n_2) -2(1-\alpha)^2\\
    &\quad-2(\nu-\alpha t_1 +1 -\alpha )\big(\lambda_i(A_{\alpha}(G_1)) -2\alpha t_1 +t_1 -2 +2\alpha\big) \Big)
\end{align*}
\endgroup

\noindent We know that $R^TR = \mathcal{L}(G_1) +2I_{m_1}$. Also, as $G_1$ is $t_1$-regular, $\mathcal{L}(G_1)$ is $(2t_1 -2)$-regular. Using these we get each row sum of $(1-\alpha)A(\mathcal{L}(G_1))+\frac{(1-\alpha)^2}{(\nu-\alpha t_1)}R^TR$ is\\ $\frac{(1-\alpha)(2t_1 -2)(\nu-\alpha t_1) + 2t_1(1-\alpha)^2}{(\nu-\alpha t_1)}$. So
\begin{align*}
    &\quad\Upsilon_{(1-\alpha)A(\mathcal{L}(G_1))+\frac{(1-\alpha)^2}{(\nu-\alpha t_1)}R^TR}(\nu-2\alpha t_1-\alpha n_2)\\
    &\quad= \frac{m_1(\nu-\alpha t_1)}{(\nu-\alpha t_1)(\nu-2\alpha t_1 -\alpha n_2)-(1-\alpha)(2t_1 -2)(\nu-\alpha t_1 +1-\alpha) -2(1-\alpha)^2}. &&
\end{align*}
Therefore
\begin{align*}
    y &=  1-  (1-\alpha)^2\Upsilon_{A_{\alpha}(G_2)}(\nu-\alpha m_1) \Upsilon_{(1-\alpha)A(\mathcal{L}(G_1))+\frac{(1-\alpha)^2}{(\nu-\alpha t_1)}R^TR}(\nu-2\alpha t_1-\alpha n_2)\notag\\
    &=\frac{\begin{pmatrix}(\nu-\alpha t_1)(\nu-2\alpha t_1 -\alpha n_2) -(1-\alpha)(2t_1 -2)(\nu-\alpha t_1 +1 -\alpha)\\ -2(1-\alpha)^2-m_1(1-\alpha)^2(\nu-\alpha t_1) \Upsilon_{A_{\alpha}(G_2)}(\nu-\alpha m_1)\end{pmatrix}}{(\nu-\alpha t_1)(\nu-2\alpha t_1 - \alpha n_2)-(1-\alpha)(2t_1 -2)(\nu-\alpha t_1 +1 -\alpha) -2(1-\alpha)^2} 
\end{align*}
Using the newly obtained expressions for $x$ and $y$ in \eqref{eq3000}, we find $\det S$. Using that expression for $\det S$ in \eqref{eq42}, we finally get the desired result. 
\end{proof}

\begin{corollary}\label{cor2.1}
    Consider two graphs $G_1$ and $G_2$ with respective orders $n_1$ and $n_2$, and regularities $t_1$ and $t_2$. Further, the size of $G_1$ is $m_1$. Then
    \begin{enumerate}[label={\upshape (\alph*)}]
        \item When $t_1 =1$, the $A_{\alpha}$-spectrum of $G_1 \underline{\vee}_Q G_2$ comprises:
        \begin{enumerate}[label= {\upshape (\roman*)}]
            \item one simple eigenvalue $\alpha$;
            \item the eigenvalues $\alpha + \lambda_j( A_{\alpha}(G_2) )$ for $j=2,3, \ldots, n_2$ and
            \item three eigenvalues obtained by solving
            \begin{equation*}\label{eq56}
                (\nu-\alpha)(\nu-2\alpha -\alpha n_2)(\nu-\alpha -t_2) -2(1-\alpha)^2(\nu-\alpha -t_2) -(1-\alpha)^2(\nu-\alpha)n_2=0.
            \end{equation*}
        \end{enumerate}

        \item When $t_1 \geq 2$, the $A_{\alpha}$-spectrum of $G_1 \underline{\vee}_Q G_2$ comprises:
        \begin{enumerate}[label= {\upshape (\roman*)}]
            \item one eigenvalue $\alpha(2t_1+n_2+2) -2$ with multiplicity $m_1 -n_1$;
            \item the eigenvalues $\alpha m_1 + \lambda_j(A_{\alpha}(G_2))$ for $j=2,3, \ldots, n_2$;
            \item eigenvalues obtained by solving each of the equations,
            \begin{align*}
                {\nu}^2 &-\big(\alpha(t_1 +n_2 +2) +t_1 -2 +\lambda_i(A_{\alpha}(G_1))\big)\nu \\
                &+\alpha^2 t_1 n_2 +\big(\alpha(t_1 +1)-1\big)\big(t_1 +\lambda_i(A_{\alpha}(G_1))\big)=0 \enskip \text{for }  i= 2,3,\ldots,n_1 \enskip \text{and}
            \end{align*}
            \item three eigenvalues obtained by solving the equation,            \begin{align*}
                &(\nu-\alpha t_1)(\nu-2\alpha t_1 -\alpha n_2)(\nu-\alpha m_1 -t_2)\\ 
                &-2(1-\alpha)(t_1 -1)(\nu-\alpha t_1 +1 -\alpha)(\nu-\alpha m_1 -t_2)\\
                &-2(1-\alpha)^2(\nu-\alpha m_1 -t_2) -m_1 n_2(1-\alpha)^2(\nu-\alpha t_1)=0.
            \end{align*}
        \end{enumerate}
    \end{enumerate}
\end{corollary}

\begin{proof}
    The proof follows a similar approach as the proof of Corollary $\ref{cor1.1}$.
\end{proof}

\begin{corollary}\label{cor2.2}
    Consider a graph $G_1$ of order $n_1$, size $m_1$ and regularity $t_1$. Let $G_2=K_{a,b}$ be the complete bipartite graph of order $n_2 = a+b$. Then
   \begin{enumerate}[label={\upshape (\alph*)}]
        \item When $t_1 =1$, the $A_{\alpha}$-spectrum of $G_1 \underline{\vee}_Q G_2$ comprises:
        \begin{enumerate}[label= {\upshape (\roman*)}]
            \item one simple eigenvalue $\alpha$;
            \item one eigenvalue $\alpha (1+a)$ with multiplicity $b-1$;
            \item one eigenvalue $\alpha (1+b)$ with multiplicity $a-1$ and
            \item four eigenvalues obtained by solving the equation,
            
            {\small
            \begin{align*}
                \big( {\nu}^2 -\alpha(3+n_2)\nu + ({\alpha}^2n_2 +4\alpha &-2) \big) \big( {\nu}^2 -\alpha (2 + n_2)\nu +( {\alpha}^2 + {\alpha}^2n_2 +2\alpha ab -ab) \big)\\
                &-(1-\alpha)^2(\nu-\alpha) \big( (\nu-\alpha)n_2 -\alpha {n_2}^2 + 2ab \big)=0.
            \end{align*}
            }
        \end{enumerate}

        \item When $t_1 \geq 2$, the $A_{\alpha}$-spectrum of $G_1 \underline{\vee}_Q G_2$ comprises:
        \begin{enumerate}[label= {\upshape (\roman*)}]
            \item one eigenvalue $\alpha(2t_1+n_2 +2) -2$ with multiplicity $m_1 -n_1$;
            \item one eigenvalue $\alpha (m_1 +a)$ with multiplicity $b-1$;
            \item one eigenvalue $\alpha (m_1 +b)$ with multiplicity $a-1$;
            \item eigenvalues obtained by solving the equation,
            \begin{align*}
                {\nu}^2 &-\big( \alpha(t_1 +n_2 +2)+t_1 -2 +\lambda_i(A_{\alpha}(G_1))\big)\nu \\
                &+\alpha^2 t_1 n_2 + \big( \alpha(t_1 +1)-1\big)\big( t_1 +\lambda_i(A_{\alpha}(G_1))\big)=0 \enskip \text{for } i=2,3,\ldots,n_1 \enskip \text{and}
            \end{align*}
            \item four eigenvalues obtained by solving the equation,
            \begin{align*}
                &\Big( {\nu}^2 -(\alpha t_1 + \alpha n_2 +2t_1 -2 +2\alpha)\nu + ({\alpha}^2t_1 n_2 +2\alpha t_1^2 +2\alpha t_1 -2t_1)   \Big) \\
                &\Big( {\nu}^2 -(2\alpha m_1 + \alpha n_2)\nu +( {\alpha}^2m_1^2 + {\alpha}^2m_1n_2 +2\alpha ab -ab)  \Big)\\
                &-m_1(1-\alpha)^2(\nu-\alpha t_1)\Big((\nu-\alpha m_1)n_2 -\alpha n_2^2 +2ab \Big)=0.
            \end{align*}
        \end{enumerate}
    \end{enumerate}
\end{corollary}

\begin{proof}
    The proof follows a similar approach as the proof of Corollary $\ref{cor1.2}$.
\end{proof}

\begin{corollary}\label{cor2.3}
    \begin{enumerate}[label={\upshape (\alph*)}]
        \item Consider two regular graphs $G_1$ and $G_2$, and an arbitrary graph $H$. Then $A_{\alpha}$-cospectralness of $G_1$ and $G_2$ implies that $G_1\underline{\vee}_Q H$ and $G_2 \underline{\vee}_Q H$ are $A_\alpha$-cospectral.
        \item Consider a regular graph $G$. Suppose $H_1$ and $H_2$ are two graphs satisfying the condition $\Upsilon_{A_{\alpha}(H_1)}(\nu) = \Upsilon_{A_{\alpha}(H_2)}(\nu)$. Then $A_{\alpha}$-cospectralness of $H_1$ and $H_2$ implies that $G\underline{\vee}_Q H_1$ and $G \underline{\vee}_Q H_2$ are $A_\alpha$-cospectral.
    \end{enumerate}
\end{corollary}


\section{$A_{\alpha}$-characteristic polynomial of $T$-vertex join} \label{sec6}

\begin{theorem}\label{th101}
    Consider two graphs $G_1$ and $G_2$ with respective orders $n_1$ and $n_2$. Let $m_1$ be the size and $t_1$ be the vertex regularity of $G_1$. Then
    \begin{align} \label{eq4001}
         \resizebox{.15\textwidth}{!}{$\phi_{A_{\alpha}(G_1 \dot{\vee}_T G_2)}(\nu)$} &\resizebox{.45\textwidth}{!}{$=(\nu +2 -2\alpha -2\alpha t_1)^{m_1 -n_1} \phi_{A_{\alpha}(G_2)}(\nu- \alpha n_1)$}\notag\\ 
         &\quad \prod_{i=2}^{n_1}\resizebox{.5\textwidth}{!}{$\bigg({\nu}^2 +\Big(2-t_1-2\alpha -\alpha t_1 -\alpha n_2 -2\lambda_i(A_\alpha(G_1))\Big)\nu$}\notag\\
        &\quad \resizebox{.82\textwidth}{!}{$-(1-\alpha)\Big(t_1 -2\alpha t_1 +\lambda_i(A_\alpha(G_1))\Big)
        +\Big(\alpha t_1 +\alpha n_2 + \lambda_i(A_{\alpha}(G_1)\Big)\Big(t_1 -2+2\alpha + \lambda_i(A_{\alpha}(G_1))\Big) \bigg)$}\notag\\
        &\quad \resizebox{.55\textwidth}{!}{$\bigg( (\nu-t_1-\alpha t_1 -\alpha n_2)(\nu+2-2\alpha -2t_1) -2t_1(1-\alpha)^2$}\notag\\
        &\quad \resizebox{.5\textwidth}{!}{$-n_1(1-\alpha)^2(\nu+2-2\alpha-2t_1) \Upsilon_{A_{\alpha}(G_2)}(\nu-\alpha n_1) \bigg)$}.
    \end{align}
\end{theorem}

\begin{proof}
We have
\begin{equation}\label{eq4002}
    A(G_1 \dot{\vee}_T G_2) =
    \begin{bmatrix}
        A(G_1) & R & J_{n_1 \times n_2}\\
        R^T & A(\mathcal{L}(G_1)) & 0_{m_1 \times n_2}\\
        J_{n_2 \times n_1} & 0_{n_2 \times m_1} & A(G_2)
    \end{bmatrix}.
\end{equation}
and
\begin{equation}\label{eq4003}
    D(G_1 \dot{\vee}_T G_2) =
    \begin{bmatrix}
        (2t_1 + n_2)I_{n_1} & 0_{n_1 \times m_1} & 0_{n_1 \times n_2}\\
        0_{m_1 \times n_1} & 2t_1I_{m_1} & 0_{m_1 \times n_2}\\
        0_{n_2 \times n_1} & 0_{n_2 \times m_1} & D(G_2) + n_1I_{n_2}
    \end{bmatrix}.
\end{equation}
Using \eqref{eq4002} and \eqref{eq4003}, we get
\small
\begin{equation*}\label{eq4004}
    A_{\alpha}(G_1 \dot{\vee}_T G_2)=
    \begin{bmatrix}
        \alpha(t_1 + n_2)I_{n_1} + A_\alpha(G_1) & (1-\alpha)R & (1-\alpha)J_{n_1 \times n_2}\\
        (1-\alpha)R^T & 2\alpha t_1I_{m_1} +(1-\alpha) A(\mathcal{L}(G_1)) & 0_{m_1 \times n_2}\\
        (1-\alpha)J_{n_2 \times n_1} & 0_{n_2 \times m_1} & \alpha n_1 I_{n_2} + A_{\alpha}(G_2)
    \end{bmatrix}.
\end{equation*}
Thus
\begin{align}\label{eq4005}
    &\quad \phi_{A_{\alpha}(G_1  \dot{\vee}_T G_2)}(\nu)\notag\\ 
    &\quad= \det \big(\nu I_{n_1+m_1+n_2}-A_{\alpha}(G_1 \dot{\vee}_T G_2) \big)\notag\\
    &\quad={\scriptsize \det \begin{bmatrix}
        (\nu-\alpha(t_1 + n_2)) I_{n_1} -A_\alpha(G_1) & -(1-\alpha)R & -(1-\alpha)J_{n_1 \times n_2}\\
        -(1-\alpha)R^T & (\nu-2\alpha t_1) I_{m_1} -(1-\alpha) A(\mathcal{L}(G_1)) & 0_{m_1 \times n_2}\\
        -(1-\alpha)J_{n_2 \times n_1} & 0_{n_2 \times m_1} & (\nu-\alpha n_1) I_{n_2} - A_{\alpha}(G_2)
            \end{bmatrix}}\notag\\
    &\quad=  \det \big( (\nu-\alpha n_1) I_{n_2} - A_{\alpha}(G_2) \big) \det S \quad (\text{using Lemma } \ref{lem1}),
\end{align}
where
{\scriptsize
\begin{align*}
    S =  \begin{bmatrix}
        (\nu-\alpha(t_1 + n_2)) I_{n_1} - A_\alpha(G_1) - (1-\alpha)^2\Upsilon_{A_{\alpha}(G_2)}(\nu-\alpha n_1) J_{n_1 \times n_1} & -(1-\alpha)R\\
        -(1-\alpha)R^T & (\nu+2-2\alpha-2\alpha t_1 ) I_{m_1} -(1-\alpha)R^TR
    \end{bmatrix},
\end{align*}
}
after simplification.
Pre-multiplying the blocks of first row of $S$ by $-R^T$, and then adding to corresponding blocks of second row of $S$, we get
\begin{align*}
   \det S =  \det \left[\begin{smallmatrix}
        (\nu-\alpha t_1 - \alpha n_2) I_{n_1} - A_\alpha(G_1) - (1-\alpha)^2\Upsilon_{A_{\alpha}(G_2)}(\nu-\alpha n_1) J_{n_1 \times n_1} & -(1-\alpha)R\\
        -R^T\big((\nu -\alpha t_1 -\alpha n_2 +1-\alpha) I_{n_1} - A_{\alpha}(G_1) -(1-\alpha)^2 \Upsilon_{A_\alpha(G_2)} (\nu-\alpha n_1) J_{n_1 \times n_1}\big) & (\nu+2-2\alpha-2\alpha t_1 ) I_{m_1}
    \end{smallmatrix}\right].
\end{align*}
Applying Lemma \ref{lem1} to it, it becomes
\begin{align}\label{eq4006}
    \det S &= \det \Big((\nu+2-2\alpha-2\alpha t_1 ) I_{m_1}\Big) \det \bigg( (\nu-\alpha t_1 - \alpha n_2) I_{n_1} - A_\alpha(G_1) - \notag\\
    &\quad (1-\alpha)^2\Upsilon_{A_{\alpha}(G_2)}(\nu-\alpha n_1) J_{n_1 \times n_1}- \frac{1-\alpha}{x+2-2\alpha-2\alpha t_1}RR^T\Big((\nu -\alpha t_1 -\alpha n_2 +1-\alpha) I_{n_1}\notag\\
    &\quad- A_{\alpha}(G_1)-(1-\alpha)^2 \Upsilon_{A_\alpha(G_2)} (\nu-\alpha n_1) J_{n_1 \times n_1}\Big)\bigg)\notag\\
    & = (\nu+2-2\alpha-2\alpha t_1 )^{m_1 -n_1} \det \bigg( (\nu +2-2\alpha-2\alpha t_1)\Big((\nu-2\alpha t_1 -\alpha n_2)I_{n_1}\notag\\
    &\quad-(1-\alpha)A(G_1)-(1-\alpha)^2\Upsilon_{A_{\alpha}(G_2)}(\nu -\alpha n_1)J_{n_1 \times n_1}\Big)-(1-\alpha)\Big(A(G_1)+t_1 I_{n_1}\Big)\notag\\
    &\quad\Big( (\nu-2\alpha t_1 -\alpha n_2 +1-\alpha)I_{n_1} -(1-\alpha)A(G_1)-(1-\alpha)^2\Upsilon_{A_\alpha(G_2)}(\nu -\alpha n_1)J_{n_1 \times n_1}\Big) \bigg)\notag\\
    &= (\nu+2-2\alpha-2\alpha t_1 )^{m_1 -n_1} \det U,
\end{align}
where $U =(\nu +2-2\alpha-2\alpha t_1)\Big((\nu-2\alpha t_1 -\alpha n_2)I_{n_1} -(1-\alpha)A(G_1)-(1-\alpha)^2\Upsilon_{A_{\alpha}(G_2)}(\nu -\alpha n_1)J_{n_1 \times n_1}\Big)-(1-\alpha)\Big(A(G_1)+t_1 I_{n_1}\Big) \Big( (\nu-2\alpha t_1 -\alpha n_2 +1-\alpha)I_{n_1} -(1-\alpha)A(G_1)-(1-\alpha)^2 \Upsilon_{A_\alpha(G_2)} (\nu -\alpha n_1)J_{n_1 \times n_1}\Big)$. $U$ is a real matrix polynomial with variables $A(G_1)$ and $J_{n_1 \times n_1}$. Thus $U(A(G_1), J_{n_1 \times n_1})$ has the eigenvalues $\zeta_1 = U(t_1, n_1)$ and $\zeta_i = U\Big(\frac{\lambda_i(A_\alpha(G_1))-\alpha t_1}{1-\alpha}, 0\Big)$ for $i=2, 3, \ldots, n_1$. Moreover $\det U = \prod_{i=1}^{n_1}\zeta_i$. Now
\begin{align*}
    \zeta_1 &= (\nu +2-2\alpha-2\alpha t_1)\Big(\nu-2\alpha t_1 -\alpha n_2 -(1-\alpha)t_1-(1-\alpha)^2\Upsilon_{A_{\alpha}(G_2)}(\nu -\alpha n_1)n_1\Big)\\
    &\quad-(1-\alpha)(t_1+t_1) \Big( \nu-2\alpha t_1 -\alpha n_2 +1-\alpha -(1-\alpha)t_1-(1-\alpha)^2 \Upsilon_{A_\alpha(G_2)} (\nu -\alpha n_1)n_1\Big)\\
    &=  \resizebox{.93\textwidth}{!}{$(\nu -t_1 -\alpha t_1 -\alpha n_2)(\nu +2-2\alpha-2 t_1) -2t_1(1-\alpha)^2 - n_1(1-\alpha)^2(\nu +2 -2\alpha -2t_1)\Upsilon_{A_\alpha(G_2)} (\nu -\alpha n_1)$},
\end{align*}
after simplification; and
\begin{align*}
    \zeta_i &= (\nu +2-2\alpha-2 \alpha t_1)\big(\nu-2\alpha t_1 -\alpha n_2 -(\lambda_i(A_\alpha(G_1))-\alpha t_1)\big) \\
    &\quad-\big(\lambda_i(A_\alpha(G_1))-\alpha t_1 +t_1(1-\alpha)\big)\big(\nu-2\alpha t_1 -\alpha n_2 +1 -\alpha -(\lambda_i(A_\alpha(G_1))-\alpha t_1)\big)\\
    &= \nu^2 + \big(2-t_1 -2\alpha -\alpha t_1 -\alpha n_2 -2\lambda_i(A_\alpha(G_1))\big)\nu -(1-\alpha)\big(t_1-2\alpha t_1 + \lambda_i(A_\alpha(G_1))\big)\\
    &\quad + \big( \alpha t_1 + \alpha n_2 + \lambda_i(A_\alpha(G_1))\big)\big(t_1 -2 +2\alpha +\lambda_i(A_\alpha(G_1)) \big).
\end{align*}
Hence
{\footnotesize
\begin{align*}
    \det U &= \Big((\nu -t_1 -\alpha t_1 -\alpha n_2)(\nu +2-2\alpha-2 t_1) -2t_1(1-\alpha)^2 - n_1(1-\alpha)^2(\nu +2 -2\alpha -2t_1)\\
    &\quad \times \Upsilon_{A_\alpha(G_2)} (\nu -\alpha n_1)\Big) \prod_{i=2}^{n_1}\Big(\nu^2 + \big(2-t_1 -2\alpha -\alpha t_1 -\alpha n_2 -2\lambda_i(A_\alpha(G_1))\big)\nu \\
    &\quad -(1-\alpha)\big(t_1-2\alpha t_1 + \lambda_i(A_\alpha(G_1))\big) + \big( \alpha t_1 + \alpha n_2 + \lambda_i(A_\alpha(G_1))\big)\big(t_1 -2 +2\alpha +\lambda_i(A_\alpha(G_1)) \big)\Big).
\end{align*}
}
Using this expression of $\det U$ in \eqref{eq4006}, and then using \eqref{eq4005}, we obtain the result as required.
\end{proof}

\begin{corollary}\label{cor10.1}
    Consider two graphs $G_1$ and $G_2$ with respective orders $n_1$ and $n_2$, and regularities $t_1$ and $t_2$. Further, the size of $G_1$ is $m_1$. Then
    \begin{enumerate}[label={\upshape (\alph*)}]
        \item When $t_1 =1$, the $A_{\alpha}$-spectrum of $G_1 \dot{\vee}_T G_2$ comprises:
        \begin{enumerate}[label= {\upshape (\roman*)}]
            \item one simple eigenvalue $\alpha (3+n_2) -1$;
            \item the eigenvalues $2\alpha + \lambda_j( A_{\alpha}(G_2) )$ for $j=2,3, \ldots, n_2$ and
            \item three eigenvalues obtained by solving
            \begin{equation*}
                \resizebox{.87 \textwidth}{!}{$(\nu-2\alpha)(\nu-2\alpha -t_2)(\nu-1-\alpha -\alpha n_2) -2n_2(1-\alpha)^2(\nu-2\alpha) -2(1-\alpha)^2(\nu-2\alpha -t_2)=0$}.
            \end{equation*}
        \end{enumerate}

        \item When $t_1 \geq 2$, the $A_{\alpha}$-spectrum of $G_1 \dot{\vee}_T G_2$ comprises:
        \begin{enumerate}[label= {\upshape (\roman*)}]
            \item one eigenvalue $2\alpha(t_1+1) -2$ with multiplicity $m_1 -n_1$;
            \item the eigenvalues $\alpha n_1 + \lambda_j(A_{\alpha}(G_2))$ for $j=2,3, \ldots, n_2$;
            \item eigenvalues obtained by solving each of the equations,
            \begin{align*}
                {\nu}^2 &-\big(\alpha(t_1 +n_2 +2) +t_1 -2 +2\lambda_i(A_{\alpha}(G_1))\big)\nu \\
                &-(1-\alpha)\big(t_1-2\alpha t_1 + \lambda_i(A_\alpha(G_1))\big)+ \big( \alpha t_1 + \alpha n_2 + \lambda_i(A_\alpha(G_1))\big)\\
                &\quad \times \big(t_1 -2 +2\alpha +\lambda_i(A_\alpha(G_1)) \big)=0 \enskip \text{for }  i= 2,3,\ldots,n_1 \enskip \text{and}
            \end{align*}
            \item three eigenvalues obtained by solving the equation,            \begin{align*}
                &(\nu-t_2 -\alpha n_1)(\nu- t_1 -\alpha t_1 -\alpha n_2)(\nu+2 -2t_1-2\alpha)\\
                &-2t_1 (1-\alpha)^2(\nu-\alpha n_1 -t_2) -n_1 n_2(1-\alpha)^2(\nu +2 -2t_1-2\alpha)=0.
            \end{align*}
        \end{enumerate}
    \end{enumerate}
\end{corollary}

\begin{corollary}\label{cor10.2}
    Consider a graph $G_1$ of order $n_1$, size $m_1$ and regularity $t_1$. Let $G_2=K_{a,b}$ be the complete bipartite graph of order $n_2 = a+b$. Then
   \begin{enumerate}[label={\upshape (\alph*)}]
        \item When $t_1 =1$, the $A_{\alpha}$-spectrum of $G_1 \dot{\vee}_T G_2$ comprises:
        \begin{enumerate}[label= {\upshape (\roman*)}]
            \item one simple eigenvalue $\alpha (3+n_2) -1$;
            \item one eigenvalue $\alpha (2+a)$ with multiplicity $b-1$;
            \item one eigenvalue $\alpha (2+b)$ with multiplicity $a-1$ and
            \item four eigenvalues obtained by solving the equation,
            \begin{align*}
                \big( {\nu}^2 &-(3\alpha +\alpha n_2+1)\nu + 2{\alpha}^2n_2 +6\alpha -2 \big) \big( {\nu}^2 -(4\alpha +\alpha n_2)\nu + 4{\alpha}^2 \\
                &+ 2{\alpha}^2n_2+2\alpha ab -ab \big)-(1-\alpha)^2(\nu-2\alpha) \big( (\nu-2\alpha)n_2 -\alpha {n_2}^2 + 2ab \big)=0.
            \end{align*}
        \end{enumerate}

        \item When $t_1 \geq 2$, the $A_{\alpha}$-spectrum of $G_1 \dot{\vee}_T G_2$ comprises:
        \begin{enumerate}[label= {\upshape (\roman*)}]
            \item one eigenvalue $2\alpha(t_1+2) -2$ with multiplicity $m_1 -n_1$;
            \item one eigenvalue $\alpha (n_1 +a)$ with multiplicity $b-1$;
            \item one eigenvalue $\alpha (n_1 +b)$ with multiplicity $a-1$;
            \item eigenvalues obtained by solving the equation,
            \begin{align*}
                {\nu}^2 &-\big( \alpha(t_1 +n_2 +2)+t_1 -2 +2\lambda_i(A_{\alpha}(G_1))\big)\nu \\
                &-(1-\alpha)\big(t_1-2\alpha t_1 + \lambda_i(A_\alpha(G_1))\big) + \big( \alpha t_1 + \alpha n_2 + \lambda_i(A_\alpha(G_1))\big)\\
                &\times \big(t_1 -2 +2\alpha +\lambda_i(A_\alpha(G_1)) \big)=0 \enskip \text{for } i=2,3,\ldots,n_1 \enskip \text{and}
            \end{align*}
            \item four eigenvalues obtained by solving the equation,
            \begin{align*}
                &\big( {\nu}^2 -(\alpha t_1 + \alpha n_2 +3t_1 -2 +2\alpha)\nu + (2\alpha^2 n_2 +2\alpha t_1^2 + 2\alpha t_1 n_2 +4\alpha t_1 \\
                &-2\alpha n_2 +2t_1^2 -4t_1)   \big) \big( {\nu}^2 -(2\alpha n_1 + \alpha n_2)\nu +( {\alpha}^2n_1^2 + {\alpha}^2n_1n_2 +2\alpha ab -ab)  \big)\\
                &-n_1(1-\alpha)^2(\nu+2-2t_1-2\alpha)\big((\nu-\alpha n_1)n_2 -\alpha n_2^2 +2ab \big)=0.
            \end{align*}
        \end{enumerate}
    \end{enumerate}
\end{corollary}

\begin{corollary}\label{cor10.3}
    \begin{enumerate}[label={\upshape (\alph*)}]
        \item Consider two regular graphs $G_1$ and $G_2$, and an arbitrary graph $H$. Then $A_{\alpha}$-cospectralness of $G_1$ and $G_2$ implies that $G_1\dot{\vee}_T H$ and $G_2 \dot{\vee}_T H$ are $A_\alpha$-cospectral.
        \item Consider a regular graph $G$. Suppose $H_1$ and $H_2$ are two graphs satisfying the condition $\Upsilon_{A_{\alpha}(H_1)}(\nu) = \Upsilon_{A_{\alpha}(H_2)}(\nu)$. Then $A_{\alpha}$-cospectralness of $H_1$ and $H_2$ implies that $G\dot{\vee}_T H_1$ and $G \dot{\vee}_T H_2$ are $A_\alpha$-cospectral.
    \end{enumerate}
\end{corollary}


\section{$A_{\alpha}$-characteristic polynomial of $T$-edge join} \label{sec7}

\begin{theorem}\label{th102}
    Consider two graphs $G_1$ and $G_2$ with respective orders $n_1$ and $n_2$. Let $m_1$ be the size and $t_1$ be the vertex regularity of $G_1$. Then
    \begin{align} \label{eq4007}
        \resizebox{.16\textwidth}{!}{$\phi_{A_{\alpha}(G_1 \underline{\vee}_T G_2)}(\nu)$} &\resizebox{.56\textwidth}{!}{$=(\nu +2 -2\alpha -2\alpha t_1 -\alpha n_2)^{m_1 -n_1} \phi_{A_{\alpha}(G_2)}(\nu- \alpha m_1)$}\notag\\ 
        &\quad \prod_{i=2}^{n_1}\resizebox{.75\textwidth}{!}{$\Big({\nu}^2 +\big(2-t_1-2\alpha -\alpha t_1 -\alpha n_2 -2\lambda_i(A_\alpha(G_1))\big)\nu -t_1(1-\alpha)(1-3\alpha)$}\notag\\
        &\quad \resizebox{.6\textwidth}{!}{$
        +\big(\alpha t_1 + \lambda_i(A_{\alpha}(G_1))\big)\big(t_1 -3+3\alpha + \alpha n_2 + \lambda_i(A_{\alpha}(G_1))\big) \Big)$}\notag\\
        &\quad \times \resizebox{.58\textwidth}{!}{$\Big( (\nu-t_1-\alpha t_1)(\nu+2-2\alpha -2t_1-\alpha n_2) -2t_1(1-\alpha)^2$}\notag\\
        &\quad \resizebox{.51\textwidth}{!}{$-\frac{1}{2}t_1 n_1(1-\alpha)^2(\nu-t_1-\alpha t_1) \Upsilon_{A_{\alpha}(G_2)}(\nu-\alpha m_1) \Big)$}.
    \end{align}
\end{theorem}

\begin{proof}
We have
\begin{equation}\label{eq4008}
    A(G_1 \underline{\vee}_T G_2) =
    \begin{bmatrix}
        A(G_1) & R & 0_{n_1 \times n_2}\\
        R^T & A(\mathcal{L}(G_1)) & J_{m_1 \times n_2}\\
        0_{n_2 \times n_1} & J_{n_2 \times m_1} & A(G_2)
    \end{bmatrix}
\end{equation}
and
\begin{equation}\label{eq4009}
    D(G_1 \underline{\vee}_T G_2) =
    \begin{bmatrix}
        2t_1 I_{n_1} & 0_{n_1 \times m_1} & 0_{n_1 \times n_2}\\
        0_{m_1 \times n_1} & (2t_1+n_2) I_{m_1} & 0_{m_1 \times n_2}\\
        0_{n_2 \times n_1} & 0_{n_2 \times m_1} & D(G_2) + m_1I_{n_2}
    \end{bmatrix}.
\end{equation}
Using \eqref{eq4008} and \eqref{eq4009}, we get
\small
\begin{equation*}\label{eq4010}
    A_{\alpha}(G_1 \underline{\vee}_T G_2)=
    \begin{bmatrix}
        \alpha t_1 I_{n_1} + A_\alpha(G_1) & (1-\alpha)R & 0_{n_1 \times n_2}\\
        (1-\alpha)R^T & \alpha (2t_1 +n_2) I_{m_1} +(1-\alpha) A(\mathcal{L}(G_1)) & (1-\alpha)J_{m_1 \times n_2}\\
        0_{n_2 \times n_1} & (1-\alpha) J_{n_2 \times m_1} & \alpha m_1 I_{n_2} + A_{\alpha}(G_2)
    \end{bmatrix}.
\end{equation*}
Applying Lemma \ref{lem1} on $A_{\alpha}(G_1 \underline{\vee}_T G_2)$ similarly as we did in the previous sections, we get
\begin{align}\label{eq4011}
    \phi_{A_{\alpha}(G_1  \underline{\vee}_T G_2)}(\nu) =\det \big( (\nu-\alpha m_1) I_{n_2} - A_{\alpha}(G_2) \big)\times \det S,
\end{align}
where
{\scriptsize
\begin{align*}
    S = \begin{bmatrix}
        (\nu-\alpha t_1) I_{n_1} - A_\alpha(G_1)  & -(1-\alpha)R\\
        -(1-\alpha)R^T & (\nu+2-2\alpha-2\alpha t_1 -\alpha n_2) I_{m_1} -(1-\alpha)R^TR - (1-\alpha)^2\Upsilon_{A_{\alpha}(G_2)}(\nu-\alpha m_1) J_{m_1 \times m_1}
    \end{bmatrix}.
\end{align*}
}
Pre-multiplying the blocks of first row of $S$ by $-R^T$, then adding to corresponding blocks of second row of $S$, and after that using $2J_{m_1 \times m_1} = J_{m_1 \times n_1}R$, we get
\begin{align*}
    \det S = \det \resizebox{5.4in}{.17in}{$\begin{bmatrix}
        (\nu-\alpha t_1) I_{n_1} - A_\alpha(G_1)  & -(1-\alpha)R\\
        -R^T\big((x-\alpha t_1 +1-\alpha)I_{n_1} -A_{\alpha}(G_1)\big) & (\nu+2-2\alpha-2\alpha t_1 -\alpha n_2) I_{m_1} - \frac{1}{2}(1-\alpha)^2\Upsilon_{A_{\alpha}(G_2)}(\nu-\alpha m_1) J_{m_1 \times n_1}R
    \end{bmatrix}$}.
\end{align*}
Again pre-multiplying the blocks of first row of $S$ by $-\frac{1}{2}(1-\alpha)\Upsilon_{A_{\alpha}(G_2)}(\nu-\alpha m_1)J_{m_1 \times n_1}$, then adding to corresponding blocks of second row of $S$, we get
\begin{align*}
   \det S =  \det \resizebox{5.4in}{.17in}{$\begin{bmatrix}
        (\nu-\alpha t_1) I_{n_1} - A_\alpha(G_1) & -(1-\alpha)R\\
        -R^T\big((\nu -\alpha t_1+1-\alpha) I_{n_1} - A_{\alpha}(G_1)\big) -\frac{(1-\alpha)}{2} \Upsilon_{A_\alpha(G_2)} (\nu-\alpha m_1) J_{m_1 \times n_1}\big((\nu -\alpha t_1)I_{n_1} - A_{\alpha}(G_1)\big) & (\nu+2-2\alpha-2\alpha t_1 -\alpha n_2 ) I_{m_1}
    \end{bmatrix}$}.
\end{align*}
Applying Lemma \ref{lem1} again, and then after simplification, it becomes
\begin{align*}
    \det S & = (\nu+2-2\alpha-2\alpha t_1 -\alpha n_2)^{m_1 -n_1}\times \det \Big( (\nu +2-2\alpha-2\alpha t_1 - \alpha n_2)\\
    &\quad \big((\nu-\alpha t_1)I_{n_1}-A_{\alpha}(G_1)\big)- (1-\alpha)RR^T \big( (\nu -\alpha t_1 +1-\alpha)I_{n_1} -A_{\alpha}(G_1)\big)\\
    &\quad-\frac{1}{2}(1-\alpha)^2\Upsilon_{A_{\alpha}(G_2)}(\nu -\alpha m_1)RJ_{m_1 \times n_1}\big((\nu -\alpha t_1)I_{n_1} -A_\alpha(G_1)\big)\Big).
\end{align*}
Since $G_1$ is $t_1$-regular graph, $RR^T =A(G_1) + t_1 I_{n_1}$ and $RJ_{m_1 \times n_1} = t_1 J_{n_1 \times n_1}$ hold. Using them, we finally get 
\begin{align}\label{eq4012}
    \det S = (\nu +2-2\alpha-2\alpha t_1-\alpha n_2)^{m_1-n_1} \times \det V,
\end{align}
where $V =(\nu +2-2\alpha-2\alpha t_1 -\alpha n_2)\big((\nu-\alpha t_1)I_{n_1} - \alpha t_1 I_{n_1}-(1-\alpha)A(G_1)\big)-(1-\alpha)\big(A(G_1)+t_1 I_{n_1}\big) \big( (\nu-\alpha t_1 +1-\alpha)I_{n_1} -\alpha t_1 I_{n_1}-(1-\alpha)A(G_1)\big) -\frac{1}{2}(1-\alpha)^2\Upsilon_{A_{\alpha}(G_2)}(\nu-\alpha m_1) t_1 J_{n_1 \times n_1} \big((\nu -\alpha t_1)I_{n_1} -\alpha t_1 I_{n_1} -(1-\alpha)A(G_1)\big)$. Similar to as we have seen in the proof of Theorem \ref{th101}, here $V$ is a real matrix polynomial with variables $A(G_1)$ and $J_{n_1 \times n_1}$, and $V(A(G_1), J_{n_1 \times n_1})$ has the eigenvalues $\zeta_1 = V(t_1, n_1)$ and $\zeta_i = V\Big(\frac{\lambda_i(A_\alpha(G_1))-\alpha t_1}{1-\alpha}, 0\Big)$ for $i=2, 3, \ldots, n_1$, with $\det U = \prod_{i=1}^{n_1}\zeta_i$. Simplifying $\zeta_1$ and $\zeta_i$ after making these substitutions, we get
\begin{align*}
    \zeta_1 &= \resizebox{.94\textwidth}{!}{$(\nu -t_1 -\alpha t_1)(\nu +2-2\alpha-2 t_1-\alpha n_2) -2t_1(1-\alpha)^2 - \frac{1}{2}t_1 n_1(1-\alpha)^2(\nu -\alpha t_1 -t_1)\Upsilon_{A_\alpha(G_2)} (\nu -\alpha m_1)$}
\end{align*} and
\begin{align*}
    \zeta_i &= \nu^2 + \big(2-t_1 -2\alpha -\alpha t_1 -\alpha n_2 -2\lambda_i(A_\alpha(G_1))\big)\nu \\
    &\quad -t_1(1-\alpha)(1-3\alpha) + \big( \alpha t_1 + \lambda_i(A_\alpha(G_1))\big)\big(t_1 -3 +3\alpha +\alpha n_2 +\lambda_i(A_\alpha(G_1)) \big).
\end{align*}
Hence
{\footnotesize
\begin{align*}
    \det V &= \Big((\nu -t_1 -\alpha t_1)(\nu +2-2\alpha-2 t_1-\alpha n_2) - \frac{1}{2}t_1 n_1(1-\alpha)^2(\nu -\alpha t_1 -t_1)\Upsilon_{A_\alpha(G_2)} (\nu -\alpha m_1)\\
    &\quad -2t_1(1-\alpha)^2\Big) \times\prod_{i=2}^{n_1}\Big(\nu^2 + \big(2-t_1 -2\alpha -\alpha t_1 -\alpha n_2 -2\lambda_i(A_\alpha(G_1))\big)\nu  -t_1(1-\alpha)(1-3\alpha) \\
    &\quad + \big( \alpha t_1 + \lambda_i(A_\alpha(G_1))\big)\big(t_1 -3 +3\alpha +\alpha n_2 +\lambda_i(A_\alpha(G_1)) \big)\Big).
\end{align*}
}
Using this expression of $\det V$ in \eqref{eq4012}, and then using \eqref{eq4011}, we obtain the result as required.
\end{proof}

\begin{corollary}\label{cor11.1}
    Consider two graphs $G_1$ and $G_2$ with respective orders $n_1$ and $n_2$, and regularities $t_1$ and $t_2$. Further, the size of $G_1$ is $m_1$. Then
    \begin{enumerate}[label={\upshape (\alph*)}]
        \item When $t_1 =1$, the $A_{\alpha}$-spectrum of $G_1 \underline{\vee}_T G_2$ comprises:
        \begin{enumerate}[label= {\upshape (\roman*)}]
            \item one simple eigenvalue $3\alpha -1$;
            \item the eigenvalues $\alpha + \lambda_j( A_{\alpha}(G_2) )$ for $j=2,3, \ldots, n_2$ and
            \item three eigenvalues obtained by solving
            \begin{equation*}
                \resizebox{.87 \textwidth}{!}{$(\nu-\alpha -1)(\nu-\alpha -t_2)(\nu-2\alpha -\alpha n_2) -n_2(1-\alpha)^2(\nu-\alpha -1) -2(1-\alpha)^2(\nu-\alpha -t_2)=0$}.
            \end{equation*}
        \end{enumerate}

        \item When $t_1 \geq 2$, the $A_{\alpha}$-spectrum of $G_1 \underline{\vee}_T G_2$ comprises:
        \begin{enumerate}[label= {\upshape (\roman*)}]
            \item one eigenvalue $\alpha(2t_1+2+n_2) -2$ with multiplicity $m_1 -n_1$;
            \item the eigenvalues $\alpha m_1 + \lambda_j(A_{\alpha}(G_2))$ for $j=2,3, \ldots, n_2$;
            \item eigenvalues obtained by solving each of the equations,
            \begin{align*}
                \resizebox{.025\textwidth}{!}{${\nu}^2$} &\resizebox{.85\textwidth}{!}{$-\big(\alpha(t_1 +n_2 +2) +t_1 -2 +2\lambda_i(A_{\alpha}(G_1))\big)\nu -t_1(1-\alpha)(1-3\alpha) + \big( \alpha t_1 + \lambda_i(A_\alpha(G_1))\big)$}\\
                & \resizebox{.65\textwidth}{!}{$\times\big(t_1 -3 +3\alpha +\alpha n_2 +\lambda_i(A_\alpha(G_1)) \big)=0 \enskip \text{for }  i= 2,3,\ldots,n_1 \enskip \text{and}$}
            \end{align*}
            \item three eigenvalues obtained by solving the equation,       
            \begin{align*}
                &(\nu-t_2 -\alpha m_1)(\nu- t_1 -\alpha t_1)(\nu+2 -2t_1-2\alpha -\alpha n_2)\\
                &-\frac{1}{2}t_1 n_1 n_2 (1-\alpha)^2(\nu-\alpha t_1 -t_1) -2t_1(1-\alpha)^2(\nu -t_2-\alpha m_1)=0.
            \end{align*}
        \end{enumerate}
    \end{enumerate}
\end{corollary}

\begin{corollary}\label{cor11.2}
    Consider a graph $G_1$ of order $n_1$, size $m_1$ and regularity $t_1$. Let $G_2=K_{a,b}$ be the complete bipartite graph of order $n_2 = a+b$. Then
   \begin{enumerate}[label={\upshape (\alph*)}]
        \item When $t_1 =1$, the $A_{\alpha}$-spectrum of $G_1 \underline{\vee}_T G_2$ comprises:
        \begin{enumerate}[label= {\upshape (\roman*)}]
            \item one simple eigenvalue $3\alpha -1$;
            \item one eigenvalue $\alpha (1+a)$ with multiplicity $b-1$;
            \item one eigenvalue $\alpha (1+b)$ with multiplicity $a-1$ and
            \item four eigenvalues obtained by solving the equation,
            \begin{align*}
                \big( {\nu}^2 &-\alpha (3 + n_2)\nu + {\alpha}^2n_2 -\alpha n_2+6\alpha -2 \big) \big( {\nu}^2 -\alpha(2 +n_2)\nu + {\alpha}^2(1+n_2) \\
                &+2\alpha ab -ab \big)-(1-\alpha)^2(\nu-\alpha -1) \big( (\nu-\alpha)n_2 -\alpha {n_2}^2 + 2ab \big)=0.
            \end{align*}
        \end{enumerate}

        \item When $t_1 \geq 2$, the $A_{\alpha}$-spectrum of $G_1 \underline{\vee}_T G_2$ comprises:
        \begin{enumerate}[label= {\upshape (\roman*)}]
            \item one eigenvalue $\alpha(2t_1+2+n_2) -2$ with multiplicity $m_1 -n_1$;
            \item one eigenvalue $\alpha (m_1 +a)$ with multiplicity $b-1$;
            \item one eigenvalue $\alpha (m_1 +b)$ with multiplicity $a-1$;
            \item eigenvalues obtained by solving the equation,
            \begin{align*}
                \resizebox{.023\textwidth}{!}{${\nu}^2$} &\resizebox{.65\textwidth}{!}{$-\big( \alpha(t_1 +n_2 +2)+t_1 -2 +2\lambda_i(A_{\alpha}(G_1))\big)\nu - t_1 (1-\alpha)(1-3\alpha)$}\\
                &\resizebox{.85\textwidth}{!}{$+\big(\alpha t_1 +\lambda_i(A_{\alpha}(G_1))\big)\big(t_1-3+3\alpha +\alpha n_2 + \lambda_i(A_{\alpha}(G_1))\big)=0 \enskip \text{for } i=2,3,\ldots,n_1 \enskip \text{and}$}
            \end{align*}
            \item four eigenvalues obtained by solving the equation,
            \begin{align*}
                \big( {\nu}^2 &-(\alpha( t_1 + n_2+2) +3t_1 -2)\nu -4t_1 +2t_1^2 +4\alpha t_1 +\alpha t_1 n_2 \\
                &+ 2\alpha t_1^2 + \alpha^2 t_1 n_2  \big) \big( {\nu}^2 -\alpha(2 m_1 + n_2)\nu + {\alpha}^2m_1(m_1 + n_2) +2\alpha ab -ab \big)\\
                &-\frac{1}{2}t_1n_1(1-\alpha)^2(\nu-t_1-\alpha t_1)\big((\nu-\alpha m_1)n_2 -\alpha n_2^2 +2ab \big)=0.
            \end{align*}
        \end{enumerate}
    \end{enumerate}
\end{corollary}

\begin{corollary}\label{cor11.3}
    \begin{enumerate}[label={\upshape (\alph*)}]
        \item Consider two regular graphs $G_1$ and $G_2$, and an arbitrary graph $H$. Then $A_{\alpha}$-cospectralness of $G_1$ and $G_2$ implies that $G_1\underline{\vee}_T H$ and $G_2 \underline{\vee}_T H$ are $A_\alpha$-cospectral.
        \item Consider a regular graph $G$. Suppose $H_1$ and $H_2$ are two graphs satisfying the condition $\Upsilon_{A_{\alpha}(H_1)}(\nu) = \Upsilon_{A_{\alpha}(H_2)}(\nu)$. Then $A_{\alpha}$-cospectralness of $H_1$ and $H_2$ implies that $G\underline{\vee}_T H_1$ and $G \underline{\vee}_T H_2$ are $A_\alpha$-cospectral.
    \end{enumerate}
\end{corollary}


\section{Concluding remarks}\label{sec8}
In this paper, we derived explicit expressions for the $A_\alpha$-characteristic polynomials and the corresponding spectra of the $Q$-vertex join, $Q$-edge join, $T$-vertex join and $T$-edge join of graphs when the base graph $G_1$ is regular. The results established here also offer a substantial computational simplification: they allow the $A_\alpha$-spectrum of large or structurally intricate join graphs to be determined directly from the spectra of the constituent graphs, avoiding the construction of large matrices and making spectral analysis efficient.

A noteworthy consequence of our work is that the signless Laplacian analogue results for the $Q$-vertex join and $Q$-edge join, which have not yet appeared in the literature, can be extracted from our expressions by taking $\alpha=\frac{1}{2}$. Very recently, Patil et al. \cite{spectra_t_vertex_edge_join} have computed the adjacency and signless Laplacian spectra for the $T$-vertex and $T$-edge join, and our results confirm their findings in those particular settings.

A natural direction for further work is to extend the results of this paper to the case where $G_1$ is irregular, since the lack of uniform degree structure in that setting introduces additional challenges and requires new techniques.


\section*{Statements and Declarations} 
\textbf{Competing interests:} The authors declare they have no competing interests.\par
\noindent\textbf{Availability of data and materials:} No data was used for the research described in the article.\par
\noindent\textbf{Funding information} There is no funding source.

\bibliographystyle{amsplain}
\bibliography{bibliojoin2}
\end{document}